\documentclass[preprint, 11pt]{elsarticle}

\usepackage[T1]{fontenc}
\usepackage[utf8]{inputenc}
\usepackage[hmargin=0.9in]{geometry}
\usepackage[babel,german=quotes]{csquotes}
\usepackage{graphicx}

\usepackage{caption}
\usepackage{amsmath}          % writing mathematical formulas
\usepackage{amsthm}           % writing mathematical theorems 
\usepackage{amssymb}          % writing mathematical symbols
\usepackage{bm, amsfonts}     % writing bold mathematical symbols
\usepackage{xcolor}
\usepackage{fixmath}
\usepackage{tikz}
\usepackage{bm}
\usepackage{makecell}

\usepackage{fancybox}

\usepackage{lipsum}

\usepackage{url}

\newtheorem{lemma}{Lemma}

\newtheorem{definition}{Definition}
\newtheorem{theorem}{Theorem}
\newdefinition{remark}{Remark}
\newdefinition{cor}{Corollary}
\newdefinition{example}{Example}
\newproof{pf}{Proof}

\usepackage{subcaption}
\usepackage[colorlinks=true,citecolor=blue,urlcolor=blue]{hyperref}
\graphicspath{{img/}}

\definecolor{blue1}{rgb}{0.8,0.8,1.0}
\definecolor{blue2}{rgb}{0.6,0.6,1.0}
\definecolor{blue3}{rgb}{0.4,0.4,0.8}
\definecolor{blue4}{rgb}{0.2,0.4,0.6}
\definecolor{blue5}{rgb}{0.0,0.2,0.6}

\definecolor{lowyellow}{cmyk}{0,0,0.15,0}
\definecolor{bordeaux}{rgb}{.7,.4,0.}
\definecolor{grey}{rgb}{.7,.7,.7}
\definecolor{rand}{rgb}{.0,.4,1.}
\definecolor{innen}{rgb}{0,0,0}

\newcommand {\dx} {\,{\rm d}{\mathbf x}}
\newcommand {\dt} {\,{\rm d}{t}}
\newcommand {\ds} {\,{\rm d}{\mathrm s}}

\newcommand {\bu} {{\mathbf u}}
\newcommand {\bv} {{\mathbf v}}

\newcommand {\bx} {{\mathbf x}}

  \newcommand{\R}{\mathbb{R}}
  
  \newcommand{\pd}[2]{\frac{\partial #1}{\partial #2}}

  \newcommand{\td}[2]{\frac{\mathrm d #1}{\mathrm d #2}}

\newcommand{\beq}{\begin{equation}}
\newcommand{\eeq}{\end{equation}}

\makeatletter
\def\ps@pprintTitle{%
  \let\@oddhead\@empty
  \let\@evenhead\@empty
  \def\@oddfoot{
    \footnotesize\itshape
    \hfill\today
  }%
  \let\@evenfoot\@oddfoot}
\makeatother

\newcommand{\Nis}{\mathcal N_i\backslash\{i\}}
\newcommand{\N}{\mathbb{N}}

\begin{document}

\begin{frontmatter} % statt \maketitle um email adressen anzuzeigen
  \title{Locally energy-stable finite element schemes for incompressible flow problems: Design and analysis for equal-order interpolations}

  \cortext[cor1]{Corresponding author}

\author[HH]{Hennes Hajduk}
\ead{hennes.hajduk@geo.uio.no}  

\address[HH]{Department of Geosciences, University of Oslo, Postboks 1022 Blindern, 0315 Oslo, Norway}

\author[DK]{Dmitri Kuzmin\corref{cor1}}
\ead{kuzmin@math.uni-dortmund.de}  

\address[DK]{Institute of Applied Mathematics (LS III), TU Dortmund University, Vogelpothsweg 87,
  D-44227 Dortmund, Germany}

\author[GL]{Gert Lube}
\ead{lube@math.uni-goettingen.de}

\address[GL]{Institute for Numerical and Applied Mathematics,
Georg-August-Universit\"at G\"ottingen, D-37083 G\"ottingen, Germany}

\author[PO]{Philipp \"Offner}
\ead{mail@philippoeffner.de}

\address[PO]{Department of Mathematics, TU Clausthal,
  Erzstra\ss e 1, D-38678 Clausthal-Zellerfeld, Germany}

\journal{Computers \& Fluids}

\begin{abstract}
We show that finite element discretizations of incompressible flow problems can be designed to ensure preservation/dissipation of kinetic energy not only globally but also locally. In the context of equal-order (piecewise-linear) interpolations, we prove the validity of a semi-discrete energy inequality for a quadrature-based approximation to the nonlinear convective term, which we combine with the Becker--Hansbo pressure stabilization. An analogy with entropy-stable algebraic flux correction schemes for the compressible Euler equations and the shallow water equations yields a weak `bounded variation' estimate from which we deduce the semi-discrete Lax--Wendroff consistency and convergence towards dissipative weak solutions. The results of our numerical experiments for standard test problems confirm that the method under investigation is non-oscillatory and exhibits optimal convergence behavior.
\end{abstract}
\end{frontmatter}

\section{Introduction}
\label{sec:intro}

%This work builds on \cite{swe,preprint}.
%Another great paper was recently submitted by Fehn et al. \cite{fehn2024}. See also the literature on entropy conservative/stable and kinetic energy preserving (KEP) schemes and Veldman's papers on skew-symmetric forms.
%\bigskip

High-resolution schemes for nonlinear systems of conservation laws are typically based on generalizations of algorithms originally designed to ensure strong stability in the case of a scalar equation~\cite{book}. Prominent representatives of such approaches include multidimensional flux-corrected transport (FCT) methods \cite{kuzmin2012b,kuzmin2010a,lohner1987} and local extremum diminishing (LED) alternatives \cite{jameson1993,jameson1995,lyra2002}. A common drawback of straightforward extensions to systems is the lack of rigorous theoretical justification. For example, many FCT and LED schemes for the compressible Euler equations produce essentially non-oscillatory approximations but do not guarantee positivity preservation or entropy stability.

The design and analysis of structure-preserving numerical methods for compressible flow problems have gained momentum in recent years. In particular, provably invariant domain preserving, entropy stable, and kinetic energy preserving finite volume schemes were proposed in the literature~\cite{bouchut2004,chandrashekar2013,coppola2023,hillebrand2023,jameson2008,ortleb2017,ray2016,subbareddy2009,tadmor2003}. Modern approaches to algebraic flux correction \cite{kuzmin2012a,kuzmin2012b} for finite element methods were designed to enforce preservation of invariant domains (positivity preservation)~\cite{carlier2023,guermond2018,kuzmin2020,zhang2011}, local discrete maximum principles for scalar quantities of interest
\cite{dobrev2018,lohmann2016}, and entropy stability conditions \cite{gassner2013,KuHaRu2021,pazner2019}. Subcell limiting techniques based on these criteria were designed in \cite{hajduk2021,kuzmin2020a,lin2024,pazner2021,rueda2024} for high-order finite elements. Feireisl et al. \cite{feireisl2020b,feireisl2021} introduced a general framework for proving convergence of structure-preserving schemes to dissipative measure-valued solutions of multidimensional hyperbolic systems under the assumptions of uniform boundedness and entropy stability. Examples of finite element analysis based on this theory can be found in \cite{abgrall2023convergence,preprint,lukacova2023}.

Numerical schemes for incompressible flow problems often use stabilization terms  to prevent spurious oscillations in the velocity and pressure fields. It is essential to ensure that these terms suppress numerical instabilities, while preserving fine-scale features of physical origin.
In the context of monotonically integrated large eddy simulation (MILES) of turbulent incompressible flows, numerical viscosity of an FCT-like discretization method serves as an implicit subgrid scale (SGS) model for unresolvable effects \cite{boris1992,drikakis2005,grinstein2004,grinstein2012,grinstein2007}.
As explained in \cite{benha,perot2011}, the kinetic energy preservation (KEP) property of discretized advection terms is an essential requirement for realistic SGS modeling of the energy cascade. The KEP design criterion dictates the use of (skew-)symmetry preserving discretization methods~\cite{veldman2019,vestappen2003}. Failure to ensure consistency with the local energy equations presented in \cite{duchon,perot2011} may result in numerical artifacts or instabilities. Many finite element methods for the incompressible Euler and Navier--Stokes equations guarantee global energy stability. However, we are not aware of any finite element method that is locally energy stable in the sense that a (semi-)discrete energy inequality holds for the nodal values. In our proofs of consistency, we use this particular energy stability property to derive a weak BV bound (cf.~\cite{coquel1991,preprint}) and show that the uniform boundedness of nodal velocities and pressures implies that of nodal accelerations. The local energy inequality formulated in \cite[Thm.~2]{zbMATH07042603} is an interesting theoretical result but it does not provide the estimates that we need.

According to Fehn et al. \cite[Sec.~7]{fehn2024}, the question of whether or not local energy stability is a prerequisite for weak convergence requires further investigations. As a first step toward that end, we adapt the AFC tools developed
in \cite{swe,kuzmin2020,book,KuHaRu2021} and the analysis carried out in
\cite{feireisl2020b,feireisl2021,preprint,lukacova2023} to continuous
$\mathbb{P}_1\mathbb{P}_1$
finite element discretizations of the Navier--Stokes system.
The velocity
is stabilized using a structure-preserving quadrature rule. Checkerboard
pressure~modes are avoided using the stabilization
proposed in \cite{becker}.
In contrast to standard analysis of finite element methods for
the incompressible Navier--Stokes equations, we do not assume continuity
or regularity of exact solutions. Our theoretical studies build on
the work of Duchon and Robert~\cite{duchon} and Sz\'ekelyhidi Jr.\
and Wiedemann \cite{zbMATH06102064}. We are interested in
analyzing the limit of vanishing viscosity and gaining deeper insight into
the behavior of dissipative weak solutions to incompressible flow
problems.

For a spectral (viscosity)
method \cite{lanthaler2021computation,zbMATH06497863} and a
first-order asymptotically preserving  staggered grid finite volume
(FV) scheme \cite{arun2024asymptoticpreservingfinitevolume}, it was recently demonstrated that  dissipative measure-valued (DMV) solutions of the compressible system converge to the DMV solutions of the incompressible system as the Mach number and the mesh size tend to zero. Further recent efforts in the field of numerical analysis were focused on incorporating stochastic noise into deterministic flow models to improve the regularity and provide a more realistic modeling of turbulence effects. These studies employ more general solution concepts, such as dissipative measure-valued martingale solutions or weak pathwise solutions, demonstrating their existence through the convergence of consistent numerical schemes and/or providing error estimates; see \cite{zbMATH07898893,zbMATH07726037,zbMATH07609837} and references cited therein. A common feature of existing proof techniques for  deterministic and stochastic models is the use of divergence-free test functions. To the best of the authors' knowledge, no direct proofs of convergence to generalized solutions are available for equal-order interpolations, such as the $\mathbb{P}_1\mathbb{P}_1$ finite element pairs considered in our work.

The remainder of this paper is organized as follows. In Section \ref{sec:method}, we
present our finite element method.  In Section \ref{sec:strong}, we prove local energy
stability and consistency under the assumption of strong convergence. Weak-$*$
convergence to  dissipative weak solutions is shown in Section \ref{sec:weak}.
The numerical examples of Section \ref{sec:examples} demonstrate that the proposed
scheme is, indeed, non-oscillatory and achieves optimal convergence rates. The
results are discussed and conclusions are drawn in Section \ref{sec:conclusions}.

\section{Stabilized finite element method}
\label{sec:method}

Let $\Omega\subset\R^d,\ d\in\{2,3\}$ be a Lipschitz domain. The
Navier--Stokes equations for the velocity $\bu(\bx,t)$ and pressure $p(\bx,t)$
of an incompressible fluid can be written in the skew-symmetric form
\begin{subequations}\label{eq_incompressible}
\begin{alignat}{2}
  \pd{\bu}{t}+\frac12[\nabla\cdot(\bu\otimes\bu)+\bu\cdot\nabla\bu]
  +\nabla p&=\nu\Delta\bu&\qquad&\mbox{in}\ \Omega\times(0,\infty),\\
\nabla\cdot\bu&=0&\qquad&\mbox{in}\ \Omega\times(0,\infty).
\end{alignat}
\end{subequations}
The kinematic viscosity $\nu\ge 0$ is assumed to be constant. In the case $\nu=0$,
the general Navier--Stokes system \eqref{eq_incompressible} reduces to the incompressible
Euler equations. We prescribe the initial condition
\beq
\bu(\cdot,0)=\bu_0\qquad\mbox{in}\ \Omega,
\eeq
where $\bu_0$ is a given solenoidal vector field. To avoid distinguishing between the
viscous and inviscid flow regimes, we impose
periodic boundary conditions
on the boundary $\Gamma=\partial\Omega$. Since the pressure is defined up to
a constant in the periodic case, we require that
$\int_{\Omega} p(\bx,t)\dx=0$ for $t>0$.
%\beq\label{zeromean}
%\int_{\Omega} p(\bx,t)\dx=0\qquad \forall t> 0.
%\eeq

If $\mathbf{u}$ is a smooth exact solution to
\eqref{eq_incompressible}, then the specific kinetic energy $
\eta(\bu)=\frac{|\bu |^2}2$ satisfies (cf. \cite{duchon})
\beq\label{energy}
\pd{\eta(\bu)}{t}+\nabla\cdot \mathbf{q}(\bu,p)\le 0\qquad\mbox{in}\ \Omega\times(0,\infty),
\eeq
where $\mathbf{q}(\bu,p)=(\eta(\bu)+p)\bu$. It follows 
that $\int_{\Omega}\eta(\bu(\bx,t))\dx$ is
a non-increasing function of  $t\ge 0$.

\subsection{Galerkin space discretization}

We  discretize \eqref{eq_incompressible}
in space using the $\mathbb{P}_1\mathbb{P}_1$
finite element pair\footnote{In Section~\ref{sec:examples}, we also show numerical results obtained with $\mathbb{Q}_1\mathbb{Q}_1$
approximations on quadrilateral meshes.}  on a conforming simplex mesh $\mathcal T_h$. The scalar
Lagrange basis functions associated with its
vertices $\bx_1,\ldots,\bx_{N_h}$ are denoted by $\varphi_{1},\ldots,\varphi_{N_h}$. They
span a finite element space $V_h$, are nonnegative, and have the
\emph{partition-of-unity} property
\beq\label{pu}
\sum_{j=1}^{N_h}\varphi_j(\bx)=1\qquad\forall\bx\in\bar\Omega.
\eeq
Inserting the continuous finite element approximations
\beq
\bu_h(\bx,t)=\sum_{j=1}^{N_h}\bu_j(t)\varphi_j(\bx),\qquad
p_h(\bx,t)=\sum_{k=1}^{N_h}p_k(t)\varphi_k(\bx)
\eeq
into the skew-symmetric Galerkin weak form
\begin{subequations}
\label{weakform}
\begin{align}\label{weakformu}
  \int_{\Omega}\left(
  \pd{\bu_h}{t}+\frac12[\nabla\cdot(\bu_h\otimes\bu_h)+\bu_h\cdot\nabla\bu_h]
  +\nabla p_h\right)\cdot \bv_h\dx+\nu \int_{\Omega}
  \nabla\bu_h :\nabla \bv_h\dx&=0,\\
  -\int_{\Omega}(\nabla\cdot\bu_h)q_h\dx&=0
\end{align}
\end{subequations}
of system \eqref{eq_incompressible} with vector-valued velocity test functions
$\mathbf{v}_h\in\mathbf{V}_h:=(V_h)^d$ and scalar-valued pressure test functions
$q_h\in V_h$, we obtain a spatial semi-discretization that can be written as
\begin{subequations}\label{galerkin}
\begin{align}
  \sum_{j\in\mathcal N_i}\left(m_{ij} \frac{\text{d}\bu_j}{\text{d} t}\right) &=
  -\sum_{j\in\mathcal N_i}[(a_{ij}+\nu s_{ij})\mathbf{u}_j
   +\mathbf{c}_{ij}p_j],\qquad
  i=1,\ldots,N_h,\\0&=
-\sum_{j\in\mathcal N_k}\mathbf{c}_{kj}\cdot\bu_j,\qquad
  k=1,\ldots,N_h.
\end{align}
\end{subequations}
We denote by $\mathcal N_i$ the set of node
indices $j$ such that the basis functions $\varphi_i$ and $\varphi_j$ have
overlapping supports. 
The constant coefficients of the semi-discrete problem \eqref{galerkin} are defined by
\beq
  m_{ij}=\int_{\Omega}\varphi_i\varphi_j\dx,\qquad
  \mathbf{c}_{ij}=\int_{\Omega}\varphi_i\nabla\varphi_j\dx,\qquad
  s_{ij}=\int_{\Omega}\nabla\varphi_i\cdot\nabla\varphi_j\dx.
\eeq
  The entries
  \beq
  a_{ij}=\frac12\int_{\Omega}[\varphi_i\bu_h\cdot\nabla\varphi_j
-\varphi_j\bu_h\cdot\nabla\varphi_i]\dx
\eeq
of the discrete convection operator
depend on $\bu_h$. It is easy to verify that
\beq
m_{ij}=m_{ji},\qquad  \mathbf{c}_{ij}=-\mathbf{c}_{ji},
\qquad s_{ij}=s_{ji},\qquad a_{ij}=-a_{ji}.
\eeq
\begin{remark}
If non-periodic
  boundary conditions are imposed on $\Gamma$, then
  $$\mathbf{c}_{ij}=-\mathbf{c}_{ji}+\int_{\Gamma}\varphi_i\varphi_j
  \mathbf{n}\ds,$$
  and the equations associated with boundary nodes need to be
  eliminated or modified (see, e.g., \cite{elman2014}).
  \end{remark}

In the case of non-vanishing viscosity $\nu>0$, we assume that the mesh $\mathcal T_h$ satisfies
the geometric angle conditions under which $s_{ij}\le 0$ for $j\ne i$ (see, e.g., \cite{brandts2020}).
From \eqref{pu} we deduce that
\beq\label{rowsums}
\sum_{j=1}^{N_h}\mathbf{c}_{ij}=\mathbf{0},\qquad
\sum_{j=1}^{N_h}s_{ij}=0,\qquad
\sum_{j=1}^{N_h}a_{ij}=-\frac12\int_{\Omega}
\bu_h\cdot\nabla\varphi_i\dx.
\eeq

\begin{remark}\label{remark_energy}
Substituting  $\bv_h=\bu_h$ and $q_h=p_h$ into \eqref{weakform},
we obtain the evolution equation
\begin{align*}
\td{}{t}\int_{\Omega}\eta(\bu_h)\dx =&-
\frac12\int_{\Omega}
\big[\nabla\cdot(\bu_h\otimes\bu_h)+\bu_h\cdot\nabla\bu_h\big]\cdot\bu_h\dx
   -\nu \int_{\Omega}
    \nabla\bu_h:\nabla \bu_h\dx\\
&=-\nu \int_{\Omega}  \nabla\bu_h:\nabla \bu_h\dx
\end{align*}
for the total kinetic energy. This proves the global energy stability property
$\td{}{t}\int_{\Omega}\eta(\bu_h)\dx\le 0$.
\end{remark}

\subsection{Energy and pressure stabilization}

To construct a finite element discretization that satisfies a
discrete version of  \eqref{energy},
we define
\beq
m_i=
\int_\Omega\varphi_i\dx=\sum_{j=1}^{N_h}m_{ij},\qquad
\tilde a_{ij}=
\frac{\bu_j+\bu_i}{2}\cdot\mathbf{c}_{ij}
=-\tilde a_{ji}
\eeq
and notice that
\beq\label{tildesums}
\sum_{j=1}^{N_h}\tilde a_{ij}=\frac12
\int_{\Omega}\varphi_i(\nabla\cdot\mathbf{u}_h)\dx=
-\frac12\int_{\Omega}
\bu_h\cdot\nabla\varphi_i\dx
=\sum_{j=1}^{N_h}a_{ij}
\eeq
under the assumption of periodicity, under which integration
by parts produces no boundary terms.
\medskip

We modify \eqref{galerkin} by using the approximations
$m_{ij}\approx \tilde m_{ij}:=m_i\delta_{ij}$ and $a_{ij}\approx
\tilde a_{ij}$.
For stabilization purposes, we use
 graph Laplacian operators $D^u=(d_{ij}^u)_{i,j=1}^{N_h}$ and
 $D^p=(d_{kj}^p)_{k,j=1}^{N_h}$ such that \cite{kuzmin2012a}
 \begin{subequations}
  \begin{align}\label{ddiffu}
  \sum_{j=1}^{N_h}d_{ij}^u&=0,\qquad 
  d_{ij}^u=d_{ji}^u\ge 0\qquad \forall j\ne i\in\{1,\ldots,N_h\},\\
    \sum_{j=1}^{N_h}d_{ij}^p&=0,\qquad 
  d_{ij}^p=d_{ji}^p\ge 0\qquad \forall j\ne i\in\{1,\ldots,N_h\}.\label{ddiffp}
  \end{align}
  \end{subequations}
Our stabilized version of the standard
Galerkin discretization \eqref{galerkin} reads (cf. \cite{swe})
\begin{subequations}\label{galerkin-stab}
\begin{align}\label{galerkin-staba}
  m_i\frac{\text{d}\bu_i}{\text{d} t} &=
  \sum_{j\in\mathcal N_i}[(d_{ij}^u-\tilde a_{ij}-\nu s_{ij})\mathbf{u}_j
   -\mathbf{c}_{ij}p_j],\qquad
  i=1,\ldots,N_h,\\ 0&=
  \sum_{j\in\mathcal N_k}[d_{kj}^pp_j-
  \mathbf{c}_{kj}\cdot\bu_j],\qquad
  k=1,\ldots,N_h.
\end{align}
\end{subequations}
An operator form of \eqref{galerkin-stab}, which is better suited for numerical analysis, is presented in \eqref{weakform-stab} below.

\begin{remark}
The momentum conservation error caused by our modification of \eqref{galerkin}
is zero because 
\begin{align}\label{eq:aux}
\sum_{i=1}^{N_h}\sum_{j=1}^{N_h}(a_{ij}
- \tilde a_{ij}+d_{ij}^u)\bu_j&=\sum_{i=1}^{N_h}\sum_{j=1}^{N_h}(a_{ij}
- \tilde a_{ij})\bu_j
=\sum_{i=1}^{N_h}\sum_{j=1}^{N_h}(a_{ij}
- \tilde a_{ij})(\bu_j+\bu_i)=0.
\end{align}
This follows from \eqref{tildesums},\eqref{ddiffu}, and the fact that
$a_{ij}- \tilde a_{ij}=-(a_{ji}- \tilde a_{ji})$ by definition.
\end{remark}

By default, we use $d_{ij}^u=0$ and the algebraic Becker--Hansbo pressure stabilization (\cite{becker},
cf. \cite{brezzi})
\beq\label{beckerhansbo}
d_{ij}^p=\begin{cases}\omega  m_{ij} & \mbox{if}\ j\ne i,\\
-\sum_{k\in\Nis} d_{ik}^p & \mbox{if}\ j=i\end{cases}
\eeq
depending on a parameter $\omega>0$.
This definition of $d_{ij}^p$ provides the properties required in
\eqref{ddiffp}.

\begin{remark}
  The accuracy of \eqref{galerkin-stab} can be improved using the
  framework of algebraic flux correction to adjust the perturbation
  of \eqref{galerkin} subject to appropriate constraints
  (see \cite{barrenechea2018,kuzmin2012a,book} for details).
\end{remark}

\subsection{Crank--Nicolson time discretization}

We discretize the nonlinear system \eqref{galerkin-stab} in time using the semi-implicit Crank--Nicolson
method, which is non-dissipative and second-order accurate. In the two-dimensional
($d=2$) case, we store the components of $\bu_i=(u_i,v_i)^\top$ in global vectors
${\texttt u}=(u_i)_{i=1}^{N_h}$ and $\texttt v=(v_i)_{i=1}^{N_h}$. The pressure unknowns $p_k$
are stored in a global vector $\texttt p=(p_k)_{k=1}^{N_h}$. We denote by
$M_L=(m_i\delta_{ij})_{i,j=1}^{N_h}$ the lumped counterpart of the consistent
mass matrix $M_C=(m_{ij})_{i,j=1}^{N_h}$. The components of $\mathbf c_{ij}=
(c_{ij}^1,c_{ij}^2)^\top$ are stored in the blocks $C_k=(c_{ij}^k)_{i,j=1}^{N_h},\ k=1,2$
of a discrete gradient operator. The velocity-dependent coefficients
\begin{align*}
  r_{ij}(\texttt u,\texttt v) = d_{ij}^u(\texttt u,\texttt v)
  - \tilde a_{ij}(\texttt u,\texttt v)-\nu s_{ij}
\end{align*}
are stored in the matrix $R(\texttt u,\texttt v)=(r_{ij}(\texttt u,\texttt v))_{i,j=1}^{N_h}$, which we update
at the beginning of each time step.

To advance the nodal values of numerical approximations $\bu_h^n\approx\bu_h(\cdot,t^n)$ and
$p_h^n\approx p_h(\cdot,t^n)$ in time from a level $t^n=n\Delta t,\ n\in\N_0$ to the level $t^{n+1}$, we
solve the linearized system
\begin{align*}
\begin{bmatrix}
  M_L-\frac{\Delta t}2R(\texttt u^n,\texttt v^n) & 0 & \Delta t C_1 \\
0 & M_L-\frac{\Delta t}2R(\texttt u^n,\texttt v^n) & \Delta t C_2 \\
\Delta t C_1^\top & \Delta t C_2^\top & \Delta t D^p
\end{bmatrix}
\begin{bmatrix}
\texttt u^{n+1} \\ \texttt v^{n+1} \\ \texttt p^{n+1}
\end{bmatrix}
= \begin{bmatrix}
[M_L+\frac{\Delta t}2R(\texttt u^n,\texttt v^n)] \texttt u^n \\  [M_L+\frac{\Delta t}2R(\texttt u^n,\texttt v^n)]\texttt v^n \\ 0
\end{bmatrix}.
\end{align*}
To enforce the zero mean condition for $p$ in the periodic case, we
modify the last row of the system matrix accordingly. Other types of
boundary conditions are handled as in Elman et al. \cite[Sec.~3.3]{elman2014}.

\section{Analysis of stability and consistency}
\label{sec:strong}

The structure of the spatial semi-discretization \eqref{galerkin-stab} is
similar to that of algebraic flux correction schemes for the shallow water
equations with topography~\cite{swe} and the compressible Euler equations~\cite{preprint}.
Exploiting this similarity, we now extend the analysis in
\cite[Sect. 3,4]{preprint} to the incompressible case.

\subsection{Energy stability}\label{sec:es}

In our analysis of \eqref{galerkin-stab}, the kinetic energy $\eta(\bu)=\frac{|\bu |^2}2$ and the flux $\mathbf{q}(\bu,p)=(\eta(\bu)+p)\bu$ replace 
entropy pairs $\{\eta,\mathbf{q}\}$ that are used in \cite{swe,preprint} for
hyperbolic problems.
In the following theorem, we show that the
modified finite element discretization \eqref{galerkin-stab} of the Navier--Stokes
system \eqref{eq_incompressible}
is \emph{locally energy stable}, i.e., consistent with a semi-discrete version
of the energy inequality \eqref{energy}.

\begin{theorem}
  Assume that the boundary conditions are periodic and
  \beq\label{assumTh1}
  d_{ij}^u-\nu s_{ij}\ge 0,\quad d_{ij}^p\ge 0\qquad
  \forall j\ne i\in\{1,\ldots,N_h\}.
  \eeq
  Then the validity of \eqref{galerkin-stab} implies
  \beq\label{claimTh1}
  m_i\td{\eta(\bu_i)}{t}+\sum_{j\in\Nis}2|\mathbf{c}_{ij}|\mathcal Q(\bu_i,p_i,\bu_j,p_j;\mathbf n_{ij})\le 0,
\eeq
where $\mathbf n_{ij}=\frac{\mathbf{c}_{ij}}{|\mathbf{c}_{ij}|}$ and
$\mathcal Q(\bu_L,p_L,\bu_R,p_R;\mathbf n)$ is a numerical flux consistent with
$\mathbf{q}(\bu,p)\cdot\mathbf{n}$.
\end{theorem}

  \begin{proof} 
    In view of the row sum properties \eqref{rowsums}, our semi-discrete scheme
     \eqref{galerkin-stab} can be written as
\begin{subequations}\label{galerkin-stable}
\begin{align}\label{galerkin-stablea}
  m_i \frac{\text{d}\bu_i}{\text{d} t} &=
  \sum_{j\in\mathcal N_i\backslash\{i\}}
      [(d_{ij}^u-\tilde a_{ij}-\nu s_{ij})(\mathbf{u}_j-\mathbf{u}_i)-
        \mathbf{c}_{ij}(p_j-p_i)]-\mathbf{u}_i
  \sum_{j=1}^N\tilde a_{ij}
      ,\\\label{galerkin-stableb}
      0 & = \sum_{j\in\mathcal N_i\backslash\{i\}}
      [d_{ij}^p(p_j-p_i)-(\bu_j+\bu_i)\cdot\mathbf{c}_{ij}].
\end{align}
\end{subequations}
Recalling the definition of $\tilde a_{ij}$ and using
\eqref{galerkin-stableb}, we obtain
\begin{align}\label{auxsum}
  \sum_{j=1}^N\tilde a_{ij}=
\sum_{j\in\mathcal N_i\backslash\{i\}}\tilde a_{ij}
  =\sum_{j\in\mathcal N_i\backslash\{i\}}
\frac{\bu_j+\bu_i}2\cdot\mathbf{c}_{ij}=
\sum_{j\in\mathcal N_i\backslash\{i\}}
\frac{d_{ij}^p}2(p_j-p_i).
\end{align}

By the chain rule, we have
$\td{\eta(\bu_i)}{t}=\bu_i\cdot
\frac{\text{d}\bu_i}{\text{d} t}$, where $\frac{\text{d}\bu_i}{\text{d} t}$
is the time derivative given by \eqref{galerkin-stablea}. 
The substitution of \eqref{auxsum} and
addition of  \eqref{galerkin-stableb} multiplied by
$\frac{|\mathbf{u}_i|^2}2+p_i$ yield the evolution equation
\begin{align}\label{etadot}\nonumber
  m_i\td{\eta(\bu_i)}{t} &=\bu_i\cdot
  \sum_{j\in\mathcal N_i\backslash\{i\}}\Big[(d_{ij}^u-\tilde a_{ij}-\nu s_{ij})(\bu_j-\bu_i)
     -\left(\frac{d_{ij}^p\bu_i}{2} +
     \mathbf c_{ij}\right)\left(p_j-p_i\right)\Big]\\
 &+\left(\frac{|\mathbf{u}_i|^2}2+p_i\right)\sum_{j\in\mathcal N_i\backslash\{i\}}
   [d_{ij}^p(p_j-p_i)-(\bu_j+\bu_i)\cdot\mathbf{c}_{ij}].
\end{align}

Following the analysis of entropy-stable schemes for nonlinear hyperbolic problems
\cite{swe,book,preprint,tadmor2003}, we use the decomposition $\bu_i=\frac12(\bu_j+\bu_i)
-\frac12(\bu_j-\bu_i)$ to show that \eqref{etadot} is equivalent to
\begin{align}\label{etadot2}
  m_i\td{\eta(\bu_i)}{t} =\sum_{j\in\mathcal N_i\backslash\{i\}}[F_{ij}+G_{ij}],
\end{align}
where we put all skew-symmetric terms into
\begin{align*}
  F_{ij}&=\frac{\bu_j-\bu_i}2\cdot\left[
    \tilde a_{ij}(\bu_j-\bu_i)+\mathbf c_{ij}(p_j-p_i)
    \right]
  \\
  &+
    (d_{ij}^u-\nu s_{ij})\left(\frac{|\mathbf{u}_j|^2}2-\frac{|\mathbf{u}_i|^2}2\right)
  -\frac{d_{ij}^p}2\left(\frac{|\mathbf{u}_j|^2}2+\frac{|\mathbf{u}_i|^2}2\right)
  (p_j-p_i)\\
  &+\frac12\left(\frac{|\mathbf{u}_j|^2}2+p_j+\frac{|\mathbf{u}_i|^2}2
    +p_i\right)\left[d_{ij}^p(p_j-p_i)-(\bu_j+\bu_i)\cdot\mathbf{c}_{ij}
      \right]\\
    &=\frac{\tilde a_{ij}}2|\bu_j-\bu_i|^2
    +\frac{\mathbf c_{ij}}2\cdot(\bu_j-\bu_i)(p_j-p_i)\\
    & -\frac{\bu_j+\bu_i}2\cdot\mathbf c_{ij}
  \left(\frac{|\bu_j|^2}2+p_j+\frac{|\bu_i|^2}2+p_i\right)\\
  &+(d_{ij}^u-\nu s_{ij})\left(\frac{|\bu_j|^2}2-\frac{|\bu_i|^2}2\right)
  +d_{ij}^p\left(\frac{p_j^2}2-\frac{p_i^2}2\right)
  =-F_{ji}
\end{align*}
and all symmetric ones into
\begin{align*}
  G_{ij}&=-\frac{\bu_j+\bu_i}2\cdot\left[
    \tilde a_{ij}(\bu_j-\bu_i)+\mathbf c_{ij}(p_j-p_i)
    \right]\\
  &-(d_{ij}^u-\nu s_{ij})\frac{|\bu_j-\bu_i|^2}2
 +\frac{d_{ij}^p}2\left(\frac{|\mathbf{u}_j|^2}2-\frac{|\mathbf{u}_i|^2}2\right)(p_j-p_i)
  \\
 &-\frac12\left(\frac{|\mathbf{u}_j|^2}2+p_j-\frac{|\mathbf{u}_i|^2}2-p_i\right)
    \left[d_{ij}^p(p_j-p_i)-(\bu_j+\bu_i)\cdot\mathbf{c}_{ij}
      \right]\\
  &=\left(\frac{\bu_j+\bu_i}2\cdot\mathbf{c}_{ij}
  -\tilde a_{ij}\right)\left(\frac{|\bu_j|^2}2-\frac{|\bu_i|^2}2\right)
  \\ &-(d_{ij}^u-\nu s_{ij})\frac{|\bu_j-\bu_i|^2}2 -d_{ij}^p\frac{(p_j-p_i)^2}2\\&
  =-(d_{ij}^u-\nu s_{ij})\frac{|\bu_j-\bu_i|^2}2 -d_{ij}^p\frac{(p_j-p_i)^2}2
=G_{ji}.
\end{align*}

The fluxes $F_{ij}=-F_{ji}$ have no influence on the total
kinetic energy $\sum_{i=1}^{N_h}m_i\eta(\bu_i)\approx \int_{\Omega}\eta(\bu)\dx$
because they cancel out upon summation of \eqref{etadot2} over $i$.
The local energy inequality \eqref{claimTh1} with 
$$\mathcal Q(\bu_i,p_i,\bu_j,p_j;\mathbf n_{ij})=-\frac{F_{ij}}{2|\mathbf{c}_{ij}|}$$
follows from \eqref{etadot2} since $G_{ij}\le 0$ by \eqref{assumTh1}.
The consistency requirement is met because
$\mathcal Q(\bu_i,p_i,\bu_j,p_j;\mathbf n_{ij})$ reduces to
$\mathbf{q}(\bu,p)\cdot\mathbf{n}_{ij}$ for $\bu_i=\bu=\bu_j,\ p_i=p=p_j$.
The proof of the theorem is complete.
  \end{proof}
  
 \begin{remark}
   In view of the fact that lumped-mass $\mathbb{P}_1$
    finite element approximations are equivalent to
   vertex-centered finite volume schemes \cite{selmin1993,selmin1996}, the
   weight $2\mathbf |\mathbf{c}_{ij}|$ of the semi-discrete
   inequality \eqref{claimTh1}
   corresponds to the surface area
   associated with the numerical energy flux
   $\mathcal Q(\bu_i,p_i,\bu_j,p_j;\mathbf n_{ij})$.
 \end{remark}
 
 \begin{remark}
 Proving energy stability for mixed finite element pairs (such as
 Taylor--Hood elements) is more difficult because we cannot multiply
 \eqref{galerkin-stableb} by $|\bu_i|^2/2+p_i$ to derive
 $F_{ij}$ and $G_{ij}$ as above.
\end{remark}

 The next lemma adapts
 the `weak BV estimate' obtained in \cite[Lem. 1]{preprint} to incompressible
 flows.

 \begin{lemma}\label{lemma1}
   For any finite time $T>0$, a solution of the semi-discrete problem
   \eqref{galerkin-stable} satisfies
    \beq \label{weakbvest}
 \int_0^T\sum_{i=1}^{N_h}
 \sum_{j\in\mathcal N_i\backslash\{i\}}
 \left[
\frac{d_{ij}^{\rm add}}2|\bu_j-\bu_i|^2+\frac{d_{ij}^p}2(p_j-p_i)^2
   \right]\dt=C_T,
 \eeq
 where $d_{ij}^{\rm add}=d_{ij}^u-\nu s_{ij}$ and
 $C_T=\sum_{i=1}^{N_h}m_i\eta(\bu_i(0))-
 \sum_{i=1}^{N_h}m_i\eta(\bu_i(T))$.
\end{lemma}

 \begin{proof}
   Using \eqref{etadot2} and the property that $F_{ij}=-F_{ji}$, we find that
   \beq\label{etasum}
   \sum_{i=1}^{N_h}
   m_i\td{\eta(\bu_i)}{t} =
\sum_{i=1}^{N_h}
   \sum_{j\in\mathcal N_i\backslash\{i\}}G_{ij},
   \eeq
   where
   $$
G_{ij}=-\frac{d_{ij}^{\rm add}}2|\bu_j-\bu_i|^2 -\frac{d_{ij}^p}2(p_j-p_i)^2.
$$
Integrating \eqref{etasum} from $t=0$ to $t=T$, we obtain 
\eqref{weakbvest}.
 \end{proof}

 The importance of Lemma \ref{lemma1} lies in the fact that 
 \eqref{weakbvest} provides a weak bound for the derivatives
 of the velocity and pressure approximations. Specifically, we have
 the following auxiliary result, cf.~\cite{preprint}.

\begin{lemma}\label{lemma2}
 Under the assumption that the sequence of meshes $\{\mathcal T_h\}_{h \searrow 0}$ is shape regular, there exist constants $C_1,C_2>0$ independent of $h$ such that
 $$C_1h|v_h|_{H^1(\Omega)}^2\le
 \sum_{i=1}^{N_h}\sum_{j\in\Nis}
 h^{d-1} |v_j-v_i|^2 \le C_2h|v_h|_{H^1(\Omega)}^2\qquad\forall v_h\in V_h.
 $$
\end{lemma}

\begin{proof}
  See \cite[Lem. 2.2]{guermond2016b} and \cite[Lem. 2]{preprint}.
\end{proof}

\subsection{Lax--Wendroff consistency}

The use of row-sum mass lumping on the left-hand side of \eqref{galerkin-staba} is
equivalent to approximating the $L^2(\Omega)$ scalar product
$\int_{\Omega}\pd{\mathbf{u}_h}{t}\cdot\mathbf{v}_h\dx=
(\pd{\mathbf{u}_h}{t},\mathbf{v}_h)_{L^2(\Omega)}$ of the Galerkin weak form \eqref{weakform}
by \cite[Ch. 8]{book}
\beq
\left(\pd{\mathbf{u}_h}{t},\mathbf{v}_h\right)_{L^2_h(\Omega)}=\sum_{i=1}^{N_h}m_i
\mathbf{v}_i\cdot\td{\bu_i}{t}.
\eeq
The discretized momentum equation
\eqref{weakformu} also contains the nonlinear term
\begin{align}\nonumber
a(\bu_h,\bv_h)=& \frac12\int_{\Omega}
[\nabla\cdot(\bu_h\otimes\bu_h)+\bu_h\cdot\nabla\bu_h]\cdot\bv_h
\dx+\nu \int_{\Omega}  \nabla\bu_h :\nabla \bv_h\dx
=\sum_{i=1}^{N_h}\mathbf{v}_i\cdot\sum_{j\in\mathcal N_i}
(a_{ij}+\nu s_{ij})\bu_j\\=&-\sum_{i=1}^{N_h}
\sum_{j=i+1}^{N_h}a_{ij}(\bu_j\cdot\bv_i
-\bu_i\cdot\bv_j)-\sum_{i=1}^{N_h}\sum_{j=i+1}^{N_h}
\nu s_{ij}(\bu_j-\bu_i)\cdot(\bv_j-\bv_i).
\end{align}
In our semi-discrete scheme \eqref{galerkin-stab}, we stabilize $a(\bu_h,\bv_h)$ by subtracting
\begin{align}\nonumber
s_h^u(\bu_h,\bv_h)=&\sum_{i=1}^{N_h}\bv_i\cdot
\sum_{j\in\mathcal N_i}(a_{ij}-\tilde a_{ij}+d_{ij}^u)\bu_j\\
=&-\sum_{i=1}^{N_h}\label{sugeneral}
\sum_{j=i+1}^{N_h}(a_{ij}-\tilde a_{ij})(\bu_j+\bu_i)
\cdot(\bv_j-\bv_i)\\&-\sum_{i=1}^{N_h}\sum_{j=i+1}^{N_h}
d_{ij}^u(\bu_j-\bu_i)\cdot(\bv_j-\bv_i).\nonumber
\end{align}
Here, we used the conservation property \eqref{eq:aux} to derive the second identity.
The linear term
\begin{align}\nonumber
b(\bv_h,p_h)&=\int_{\Omega}\nabla p_h\cdot\bv_h\dx
=\sum_{i=1}^{N_h}\bv_i\cdot
\sum_{j\in\mathcal N_i}\mathbf{c}_{ij}p_j=\sum_{i=1}^{N_h}\bv_i\cdot
\sum_{j=1}^{N_h}\mathbf{c}_{ij}(p_j+p_i)\\
&
=\sum_{i=1}^{N_h}
\sum_{j=i+1}^{N_h}\mathbf{c}_{ij}\cdot(\bv_i-\bv_j)
(p_j+p_i)
\end{align}
of the Galerkin momentum equation remains unchanged. The pressure
is stabilized using
\beq\label{spgeneral}
s_h^p(p_h,q_h)=\sum_{i=1}^{N_h}q_i\sum_{j\in\Nis}d_{ij}^p(p_j-p_i)
=-\sum_{i=1}^{N_h}\sum_{j=i+1}^{N_h}d_{ij}^p(p_j-p_i)(q_j-q_i).
\eeq
  Substituting the Becker--Hansbo definition \eqref{beckerhansbo}
  of $d_{ij}^p$ into \eqref{spgeneral}, we obtain
\beq\label{spbecker}
s_h^p(p_h,q_h)
=\omega[(p_h,q_h)_{L^2(\Omega)}-(p_h,q_h)_{L^2_h(\Omega)}]=-\omega
\sum_{i=1}^{N_h}\sum_{j=i+1}^{N_h}m_{ij}(p_j-p_i)(q_j-q_i).
\eeq

The semi-discrete problem associated with \eqref{galerkin-stab}
can now be written in the generic operator form
\begin{subequations}
  \label{weakform-stab}
\begin{align}
  \left(\pd{\mathbf{u}_h}{t},\mathbf{v}_h\right)_{L^2_h(\Omega)}
  +a(\bu_h,\bv_h)+b(\bv_h,p_h)
  &=s_h^u(\bu_h,\bv_h)
  \qquad\forall \bv_h\in \mathbf V_h,\label{weakform-staba}
  \\
  b(\bu_h,q_h)&=s_h^p(p_h,q_h)\qquad\forall q_h\in V_h.\label{weakform-stabb}
\end{align}
\end{subequations}

The standard finite element
interpolation operator $I_h:C(\bar\Omega)\to V_h$ approximates $v\in C(\bar\Omega)$ by
  a piecewise-linear function $I_hv=\sum_{i=1}^{N_h}v_i\varphi_i\in V_h$
  such that $v_i=v(\mathbf{x}_i)$ (see, e.g., \cite{larsson2003}).
  We say that the stabilization built into the semi-discrete scheme
  \eqref{weakform-stab}
  is \emph{Lax--Wendroff consistent} if
$$
s_h^u(\bu_h,\bv_h)\to 0,\qquad
s_h^p(p_h,q_h)\to 0\qquad \mbox{as}\ h\to 0
$$
for finite element interpolants $\bv_h=I_h\bv\in \mathbf V_h$
and $q_h=I_hq\in V_h$ of smooth test functions $\bv$ and $q$.

\begin{lemma}\label{lemma3}
 Assume that  $\|p_h\|_{L^2(\Omega)}$ is bounded by a constant independent of $h$.
  Then the pressure stabilization term $s_h^p(p_h,q_h)$ defined by \eqref{spbecker}
is Lax--Wendroff consistent.
\end{lemma}
\begin{proof}
  Following the proofs of \cite[Lem. 8.72]{book} and \cite[Cor. 8.74]{book}, we define
  $$
  e_h(q_h)=
  \sup_{p_h\in V_h}
  \frac{|(p_h,q_h)_{L^2(\Omega)}-(p_h,q_h)_{L^2_h(\Omega)}|}{\|p_h\|_{L^2(\Omega)}}\qquad
  \forall q_h\in V_h
  $$
  and use the fact that the estimate \cite[Cor. 8.74]{book}
  $$
 e_h(q_h)\le Ch|q_h|_{H^1(\Omega)}
 $$
 holds for some constant $C>0$ independent of $h$. Since
 $|s_h^p(p_h,q_h)|\le \omega \|p_h\|_{L^2(\Omega)}e_h(q_h)$, it
 follows that
$$|s_h^p(p_h,q_h)|\le C\omega h\|p_h\|_{L^2(\Omega)}|q_h|_{H^1(\Omega)},$$
 where $\|p_h\|_{L^2(\Omega)}$ is bounded by assumption and
 $|q_h|_{H^1(\Omega)}=|I_hq|_{H^1(\Omega)}
 \le C\|q\|_{H^2(\Omega)}$. This estimate of
 the error due to mass lumping
 proves the Lax--Wendroff consistency
 of $s_h^p(p_h,q_h)$.
\end{proof}

Adapting the flux correction schemes
presented in \cite{kuzmin2012a,book,preprint}
to the incompressible case, we define
\beq\label{dupw}
d_{ij}^u=\begin{cases}
(1-\alpha_{ij})d_{ij}^{\max} & \mbox{if}\ j\ne i,\\
-\sum_{k\in\Nis}d_{kj}^u  & \mbox{if}\ j=i,
\end{cases}
\eeq
where $\alpha_{ij}\in[0,1]$ is an appropriately chosen correction factor and
\beq
d_{ij}^{\max}=\begin{cases}
|a_{ij}| & \mbox{if}\ j\ne i,\\
-\sum_{k\in\Nis}d_{kj}^{\max}  & \mbox{if}\ j= i
\end{cases}
\eeq
is the artificial viscosity coefficient of the low-order
\emph{discrete upwinding} \cite{kuzmin2012a} method. Let
\beq
d_h(\bu_h,\bv_h)=\sum_{i=1}^{N_h}\sum_{j=i+1}^{N_h}
d_{ij}^u(\bu_j-\bu_i)\cdot(\bv_j-\bv_i).
\eeq
Then for $d_{ij}^{\max}=\mathcal O(h^{d-1})$ and
any choice of $\alpha_{ij}\in[0,1]$, we have
the estimate \cite{barrenechea2018}
\beq\label{dhest}
|d_h(\bu_h,\bv_h)|\le\sqrt{d_h(\bu_h,\bu_h)}\sqrt{d_h(\bv_h,\bv_h)}
\le Ch|\bu_h|_{H^1(\Omega)}|\bv_h|_{H^1(\Omega)}.
\eeq
If Lemmas \ref{lemma1}
and \ref{lemma2} are applicable, then $h^{\frac12}|\bu_h|_{H^1(\Omega)}$
is bounded and $$|d_h(\bu_h,\bv_h)|\le 
Ch^{\frac12}|\bv_h|_{H^1(\Omega)}=Ch^{\frac12}|I_h\bv|_{H^1(\Omega)}
\le
Ch^{\frac12}\|\bv\|_{H^2(\Omega)}.$$ We use this estimate of
the consistency error in the proof of the following lemma.

\begin{lemma}\label{lemma4}
  
  The velocity stabilization term $s_h^u(\bu_h,\bv_h)$ defined by \eqref{sugeneral} with
$d_{ij}^{\max}=\mathcal O(h^{d-1})$ is Lax--Wendroff consistent under the assumptions of Lemmas \ref{lemma1} and \ref{lemma2}. 
\end{lemma}
\begin{proof}
  The stabilization term \eqref{sugeneral} of the momentum
equation \eqref{weakform-staba}
  admits the decomposition
$$
s_h^u(\bu_h,\bv_h)=\sum_{i=1}^{N_h}\bv_i\cdot
\sum_{j=i}^{N_h}(a_{ij}-\tilde a_{ij})\bu_j-d_h(\bu_h,\bv_h),
$$
where
\begin{align*}
a_{ij}-\tilde a_{ij}
&=\frac12\int_{\Omega}[\varphi_i\bu_h\cdot\nabla\varphi_j
-\varphi_j\bu_h\cdot\nabla\varphi_i]\dx
-\frac{\bu_j+\bu_i}2\cdot\int_{\Omega}\varphi_i\nabla\varphi_j\dx\\
&=\frac12\int_{\Omega}[\varphi_i(\bu_h-\bu_i)\cdot\nabla\varphi_j
  -\varphi_j(\bu_h-\bu_j)\cdot\nabla\varphi_i]\dx
\end{align*}
and therefore
\begin{align*}
s_h^u(\bu_h,\bv_h)=&\frac12
\int_{\Omega}(\bv_h\otimes\bu_h-I_h(\bv_h\otimes\bu_h)):\nabla\bu_h\dx\\
-&\frac12\int_{\Omega}[
  \bu_h\otimes\bu_h-I_h(\bu_h\otimes\bu_h)]:\nabla\bv_h
\dx-d_h(\bu_h,\bv_h),\\
=&\frac12\sum_{e=1}^{E_h}\nabla\bu_h|_{K^e}:\left[
\int_{K^e}(\bv_h\otimes\bu_h-I_h(\bv_h\otimes\bu_h))\dx\right]\\
-&\frac12\sum_{e=1}^{E_h}\nabla\bv_h|_{K^e}:\left[
\int_{K^e}[
  \bu_h\otimes\bu_h-I_h(\bu_h\otimes\bu_h)]
\dx\right]-d_h(\bu_h,\bv_h).
\end{align*}
For scalar linear polynomials $u_h,v_h\in\mathbb{P}_1(K^e)$, we can show that
$$
\left|\int_{K^e}(v_hu_h-I_h(v_hu_h))\dx\right|
=|(v_h,u_h)_{L^2(K^e)}-(v_h,u_h)_{L_h^2(K^e)}|\le Ch\|u_h\|_{L^2(K^e)}
|v_h|_{H^1(K^e)}
$$
as in  Lemma \ref{lemma3}. This enables us to estimate individual
contributions of $K\in\mathcal T_h$ to $s_h^u(\bu_h,\bv_h)$. To
complete the proof, we exploit the boundedness of
$|\bv_h|_{H^1(\Omega)}=
|I_h\bv|_{H^1(\Omega)}\le C \|\bv\|_{H^2(\Omega)}$
for  smooth $\bv$, and
(similarly to \cite[Thm. 3]{preprint}) the weak BV bound for
$h^{\frac12}|\bu_h|_{H^1(\Omega)}$ provided by Lemmas \ref{lemma1} and \ref{lemma2}.
\end{proof}

\begin{lemma}\label{lemma5}
  The Galerkin part of the scheme
  \eqref{weakform-stab} is Lax--Wendroff consistent in the sense that
 \begin{gather*}
  b(\bv_h,p_h)=-\int_{\Omega} p_h(\nabla\cdot\bv_h)\dx,\qquad
  b(\bu_h,q_h)=\int_{\Omega} \bu_h\cdot\nabla q_h\dx,\\
   \left| a(\bu_h,\bv_h)-
\int_{\Omega}(\nu
\nabla\bu_h-\bu_h\otimes\bu_h):\nabla \bv_h\dx
\right|\to 0 \qquad \mbox{as}\ h\to 0.
\end{gather*}
\end{lemma}

\begin{proof}
  In view of the vector identity
  $\nabla\cdot (\bu_h\otimes\bu_h)=\bu_h\cdot\nabla\bu_h
  +\bu_h\nabla\cdot\bu_h$,
  it is sufficient to show that the consistency error
  $R_h=
  \int_{\Omega}(\nabla\cdot\bu_h)\bu_h\cdot\bv_h\dx
  $ vanishes in the limit $h\to 0$.
  By \eqref{weakform-stabb}, we have
  \begin{align*}
  R_h&=\int_{\Omega}(\nabla\cdot\bu_h) I_h(\bu_h\cdot\bv_h)\dx
  +
  \int_{\Omega}(\nabla\cdot\bu_h)[\bu_h\cdot\bv_h-
    I_h(\bu_h\cdot\bv_h)] \dx\\
  &=-s_h^p(p_h,I_h(\bu_h\cdot\bv_h))+
 \sum_{e=1}^{E_h} (\nabla\cdot\bu_h)|_{K^e}\left[
  (\bu_h,\bv_h)_{L^2(K^e)}
  - (\bu_h,\bv_h)_{L_h^2(K^e)}\right].
  \end{align*}
  To complete the proof, we invoke Lemma \ref{lemma3} and estimate the lumping error
  as in Lemma \ref{lemma4}.
  \end{proof}
  
We are now ready to formulate the main consistency result of this section. In 
the following theorem, functions belonging to the space
$C^2_c(\bar\Omega \times [0,T))$ have compact support in time.

\begin{theorem}[Lax--Wendroff consistency of the energy-stable FEM]\label{th_Lax}
  Let $T>0$ be a final time.
  Assume that Lemmas \ref{lemma1}--\ref{lemma5} are applicable, and a
  sequence of approximations $(\bu_h,p_h)$ obtained with
 \eqref{weakform-stab}
  converges to a limit 
  $(\bu,p)$ strongly. Then $(\bu,p)$ is a weak solution of problem
  \eqref{eq_incompressible} in the sense that
\begin{subequations}
\label{weaksolu}
\begin{align}  \nonumber
\int_0^T\int_{\Omega}\left[\bu\cdot\pd{\bv}{t}+
  (\bu\otimes\bu-\nu
  \nabla\bu):\nabla \bv+ p\nabla\cdot\bv\right]\dx\dt
+\int_{\Omega}\bu(\bx,0)\cdot\bv(\bx,0)\dx&=0,\\
\int_{\Omega} \bu\cdot \nabla q&=0
\end{align}
holds for all test functions $\mathbf{v}\in [C^2_c(\bar\Omega \times [0,T))]^d$
and $q\in C^1(\bar\Omega)$
that are periodic in space.
\end{subequations}
  
\end{theorem}  

\begin{proof}
  Let $\bv_h=I_h\bv$ and $q_h=I_h q$ for test functions $\bv$ and $q$ that meet
  the assumptions of the theorem.
  The consistency error due to mass lumping in the time derivative term
  can be estimated as in \cite[Thm. 3]{preprint}. The Lax--Wendroff 
  consistency of the remaining terms can be shown using
  Lemmas \ref{lemma1}--\ref{lemma5}.
\end{proof}

\section{Convergence analysis}
\label{sec:weak}

In this section, we adapt the convergence theory developed in \cite[Sec. 5]{preprint} for flux-corrected finite element discretizations of the compressible Euler equations to the incompressible case, i.e., to the vanishing viscosity limit of problem \eqref{eq_incompressible}. The existence and uniqueness of a classical solution to
\begin{equation}\label{eq_incompressible_Euler}
\begin{aligned}
  \pd{\bu}{t}+\frac12[\nabla\cdot(\bu\otimes\bu)+\bu\cdot\nabla\bu]
  +\nabla p&=0&\qquad&\mbox{in}\ \Omega\times(0,\infty),\\
\nabla\cdot\bu&=0&\qquad&\mbox{in}\ \Omega\times(0,\infty), \\
\bu( \cdot, 0)= \bu_0, \qquad \nabla \cdot\bu_0&=0&\qquad&\mbox{in}\ \Omega
\end{aligned}
\end{equation}
can be guaranteed under suitable assumptions on the initial conditions \cite{kato1984,majda1984}, e.g., for $\bu_0 \in H^k$. If these assumptions are violated, weak solutions of the incompressible Euler equations \eqref{eq_incompressible_Euler} may develop singularities, such as blow-ups in vorticity. The framework of classical and weak solutions is insufficient for an adequate mathematical description of turbulent or highly irregular flows. Therefore, the concept of measure-valued solutions was introduced by DiPerna  \cite{diperna1985compensated} already in the 1980s. Following a more recent trend  \cite{arun2024asymptoticpreservingfinitevolume,brenier2011, lanthaler2021computation,wiedemann2017weak}, we analyze \textit{dissipative measure-valued} (DMV) solutions in this article. In situations where the classical solution concept breaks down, DMV solutions provide an alternative that handles energy dissipation, oscillations and concentrations in a physics-compatible manner.

To give a formal definition of dissipative measure-valued solutions to the incompressible Euler system \eqref{eq_incompressible_Euler}, we first introduce the necessary notation. Let \( \tilde{\bu} \) represent a generic element of the phase space \( \mathcal{F}_{\text{inc}} = \mathbb{R}^d \). The space of all probability measures on \( \mathcal{F}_{\text{inc}} \) is denoted by  $\mathcal{P}(\mathcal{F})$ in what follows.  The notation $\mathcal{M}^+(\overline{\Omega})$ is used for the set of all positive  \textit{Radon measures} that can be identified with the space of all linear forms on $C_c(\overline{\Omega}).$ If $\overline{\Omega}$ is compact, then
$[C_c(\overline{\Omega})]^*=\mathcal{M}(\overline{\Omega})$. Furthermore, we denote by $\mathcal{M}(\overline{\Omega},\R^{d\times d}_{\mathrm{sym}})$ the set of symmetric matrix-valued measures and by $\mathcal{M}^+(\overline{\Omega}; \R^{d\times d}_{\mathrm{sym}})$ the set of symmetric positive-definite matrix-valued measures, i.e.,
\begin{align*}
\mathcal{M}^+(\overline{\Omega}, \R^{d\times d}_{\mathrm{sym}}) = \bigg\{& \nu \in \mathcal{M}(\overline{\Omega}, \R^{d\times d}_{\mathrm{sym}})
\bigg|
\int_{\overline{\Omega}} \phi(\xi \otimes\xi): \operatorname{d} \nu\geq0  \text{ for any } \xi \in \R^d, \phi \in [C_c(\overline{\Omega})]^{d\times d}, \phi\geq 0
 \bigg\}.
\end{align*}
We are now ready to list the properties that define a DMV solution
in the work of Arun et al. \cite{arun2024asymptoticpreservingfinitevolume}.
\begin{definition}[Dissipative measure-valued solution]\label{def_DMV}
A family of probability measures $$ \mathcal U = \{ \mathcal U_{\mathbf{x},t} \}_{(\mathbf{x},t) \in  \Omega \times (0,T)} \in L^\infty_{\mathrm{weak-}*}( \Omega \times (0, T); \mathcal{P}(\mathcal{F}_{\mathrm{inc}})) $$ is a dissipative measure-valued solution of \eqref{eq_incompressible_Euler} with initial data \( \bu_0 \) if the following conditions hold:
\begin{itemize}
    \item \textbf{Energy inequality} -- The integral inequality
    \begin{equation}
    \label{3.5}
    \frac{1}{2} \int_{\Omega} \langle \mathcal U_{\mathbf{x},t}; |\tilde{\bu}|^2 \rangle \, \dx + \int_{\Omega}  \mathrm{d} \mathcal{D}_{cd}(t) \leq \frac{1}{2} \int_{\Omega} |\bu_0|^2 \, \dx
    \end{equation}
    holds for almost every \( t \in (0, T) \) with \( D_{cd} \in L^\infty(0, T; \mathcal{M}^+(\Omega)) \).
    
     \item \textbf{Divergence-free condition} -- The integral identity
    \begin{equation}\label{eq_div_free}
    \int_{\Omega} \langle \mathcal{U}_{\mathbf{x},t}; \tilde{\bu} \rangle \cdot \nabla q\, \dx = 0
    \end{equation}
    holds for almost every \( t \in (0, T) \) and for all \( q\in C_c^\infty(\Omega) \).
    
    \item \textbf{Momentum equation} -- The integral identity
    \begin{equation}\label{moment}
    \begin{aligned}
   & \int_0^T \int_{\Omega} \left( \langle \mathcal{U}_{\mathbf{x},t}; \tilde{\bu} \rangle \cdot \partial_t \bv+ \langle \mathcal{U}_{\mathbf{x},t}; \tilde{\bu} \otimes \tilde{\bu} \rangle : \nabla \bv\right) \, \dx \, \dt \\
    &+ \int_0^T \int_{\Omega} \nabla \bv: \mathrm{d} \mathfrak{M}_{cd}(t) \, \dt + \int_{\Omega} \bu_0 \cdot \bv( \cdot, 0) \, \dx = 0
    \end{aligned}
    \end{equation}
    holds for all \( \bv \in C_c^\infty(\Omega \times  [0, T) ; \mathbb{R}^d) \) with \( \nabla \cdot \bv = 0 \), where \( \mathfrak{M}_{cd} \in L^\infty(0, T; \mathcal{M}(\Omega; \mathbb{R}^{d \times d})) \).
    
    \item \textbf{Defect compatibility condition} -- There exists \( c > 0 \) such that for almost every \( t \in (0, T) \)
    \begin{equation}
    \label{3.8}
    |\mathfrak{M}_{cd}(t)|(\Omega) \leq c \mathcal{D}(t),
    \end{equation}
where \( \mathcal{D}(t) = \int_{\Omega} \mathrm{d} \mathcal{D}_{cd}(t) > 0 \in L^\infty(0, T) \) denotes the dissipation defect.
\end{itemize}
\end{definition}
\begin{remark}
  A definition of DMV solutions was recently given in \cite{gwiazda2020dissipative} for general systems of conservation laws. In particular, it
  is valid for the compressible and incompressible Euler equations, as
  well as for the shallow water magnetohydrodynamics equations.
 Moreover, a weak-strong uniqueness result was obtained in \cite{gwiazda2020dissipative}. Focusing on this particular definition, we remark that any DMV solution to the incompressible Euler system is generated by a sequence of energy-admissible weak solutions $\{\bu_n \}_{n\in \N}$ of that system \cite{zbMATH06102064}. The dissipation defect and the defect in the momentum arise because the weak-$*$ limit of corresponding nonlinear terms does not coincide with the classical limit. We have
  \[
\mathcal{D}_{cd}(t) = \overline{\eta(\bu)}(t, \cdot) - \frac{1}{2} \langle \mathcal U_{\mathbf{x},t}; | \tilde{\bu} |^2 \rangle,\qquad
%\]
%\[
\mathfrak{M}_{cd} (t) = \overline{\bu \otimes \bu}(t, \cdot) - \langle \mathcal U_{(\mathbf{x},t)}; \tilde{\bu} \otimes \tilde{\bu} \rangle,
\]
where \( \overline{\eta(\bu)}\) is the weak-$*$ limit of \(\eta(\bu_n)=\frac{1}{2} |\bu_n|^2 \in L^\infty(0, T; \mathcal{M}^+(\Omega))\) for \(n \in \mathbb{N}\) and \(\bu \otimes \bu\) is the \mbox{weak-$*$} limit of \(\{\bu_n \otimes \bu_n\}_{n \in \mathbb{N}}\) in \(L^\infty(0, T; \mathcal{M}(\Omega; \mathbb{R}^{d \times d}))\).
In this context, weak-strong uniqueness implies that the defect measures disappear for a classical solution to the incompressible Euler system,  and the probability measure becomes a $\delta$-distribution, as demonstrated in \cite{wiedemann2017weak}.
\end{remark}
Following our analysis of flux-corrected finite element discretizations of the compressible Euler equations~\cite{preprint}, we now adapt the consistency result established in Theorem~\ref{th_Lax} as follows.

\begin{theorem}[Consistency in the vanishing viscosity limit]\label{th_consistency}
Assume that Lemmas \ref{lemma1}--\ref{lemma5} are applicable on the time interval $[0,T]$ and let the initial data $\bu_{h,0}$ be admissible.
Denote by $\mathbf{U}_h=(\bu_h, p_h)$ a uniformly bounded numerical solution to system \eqref{eq_incompressible_Euler} obtained with the scheme \eqref{weakform-stab}. Define the consistency errors~$e_{q_h}(q,t)$ and
$e_{\mathbf{u}_h} (\bv,t)$  w.r.t. test functions $q$ and $\bv$, respectively, by the property that
\begin{itemize} 
\item For any $q \in C_c^{1}(\Omega)$ and a.e. $t\in (0,T)$
\begin{equation}\label{divergence_free}
\int_{\Omega} \bu_{h} \cdot \nabla q \dx = e_{q_h}(q,t).
\end{equation}
\item For any $ \bv\in C^{2}_c( [0,T) \times \Omega;\R^d)$ with $\nabla \cdot \bv=0$ and any $\tau\in(0,T]$
  \begin{equation}\label{eq:consistency_m}
  \begin{aligned}
 \left[  \int_{\Omega} \bu_h\cdot  \bv \dx \right]_{t=0}^{t=\tau} =&\int_0^\tau \int_{\Omega}\left[ \bu_h\cdot \pd{\bv}{t}+(\bu_h \otimes \bu_h) : \nabla\bv\right] \dx \dt + \int_0^\tau e_{\mathbf{u}_h} (\bv,t) \dt.
 \end{aligned}
  \end{equation}
\end{itemize}
Then the semi-discrete scheme \eqref{weakform-stab} is consistent with the weak form of the incompressible Euler equations \eqref{eq_incompressible_Euler} in the sense that $\|e_{q_h}\|_{L^1(0,T)}\to 0$ and $\|e_{\bu_h}\|_{L^1(0,T)}\to 0$ as
$h\to 0$.
\end{theorem}
\begin{proof}
  To derive \eqref{divergence_free} following the proof of
  Lemma \ref{lemma5}, we use the splitting $q=q_h+(q-q_h)$,
  \mbox{where $q_h= I_h q$} is the interpolant of the smooth test function $q$.
  Using the definition \eqref{divergence_free} of the consistency error
  $e_{q_h}(q,t)$ and integration by parts under the assumption of periodicity,
  we find that
\begin{align*}\label{divergence_free}
 e_{q_h}(q,t)=-
  \int_{\Omega} q  \nabla \cdot \bu_{h}  \dx =
 \int_{\Omega}  \underbrace{\left(  q_h- q \right)}_{\mathcal{O}(h^2)}  (\nabla \cdot \bu_{h}  ) \dx
 -  \int_{\Omega}    q_h     (\nabla \cdot \bu_{h}) \dx
\stackrel{\eqref{weakform-stabb}}{=}
 \mathcal{O}(h^2)+s_h^p(p_h, q_h).
 \end{align*}
The application of Lemma \ref{lemma3} proves that
$e_{q_h}(q,t)$ vanishes in the limit $h\to 0$.
\smallskip

To prove the consistency of
\eqref{eq:consistency_m} with the momentum equation, we first notice that the identity
\begin{equation*}
 \left[  \int_{\Omega} \bu_h\cdot  \bv \dx \right]_{t=0}^{t=\tau} = \int_0^\tau \int_{\Omega} \frac{\partial}{\partial t} (\bu_h \cdot  \bv) \dx \dt=
  \int_0^\tau \int_{\Omega} \left[ \bu_h\cdot \pd{\bv}{t} +\pd{\bu_h}{t} \cdot \bv\right] \dx \dt
\end{equation*}
holds for all $ \bv\in C^{2}_c(\Omega\times[0,T);\R^d) $. Using the interpolation error estimate for the difference between the smooth divergence-free test function $\bv $ and its interpolant $\bv_h=I_h\bv$, we find that
\begin{equation}\label{remark_equation}
 \int_{\Omega}  \pd{\bu_h}{t}\cdot\bv
  \dx =   \int_{\Omega} \pd{\bu_h}{t}\cdot
\underbrace{(\bv- \bv_h)}_{ \mathcal{O}(h^2)} \dx  +
   \int_{\Omega}  \pd{\bu_h}{t}\cdot \bv_h  \dx.
\end{equation}
The nodal states $\bu_i(t) = \bu_h(\bx_i, t)$ are uniformly bounded by assumption and satisfy the semi-discrete energy inequality \eqref{claimTh1}. The Lipschitz continuity of the numerical flux function $\mathcal Q$ implies that the time derivative of $\bu_h$ is uniformly bounded as well. Therefore, the first term on the right-hand side of \eqref {remark_equation} vanishes as $h \to 0$.
To estimate the second term, we follow the proof of
Lax-Wendroff consistency in Theorem \ref{th_Lax}. Invoking \eqref{weakform-staba} and using the assumption that $\nabla\cdot\bv=0$, we arrive at
\begin{align*}
\int_{\Omega}  \pd{\bu_h}{t} \cdot\bv
\dx ={}&
\int_{\Omega} (\bu_h \otimes \bu_h): \nabla \bv \dx + e_{\mathbf{u}_h} (\bv,t),
\end{align*}
where
\begin{align*}
  e_{\mathbf{u}_h} (\bv,t) ={}&\int_{\Omega} \pd{\bu_h}{t}\cdot
(\bv- \bv_h) \dx +
  \left(\pd{\bu_h}{t},\bv_h\right)_{L^2(\Omega)}-\left(\pd{\bu_h}{t},\bv_h\right)_{L_h^2(\Omega)}\\
    {}&- \int_{\Omega} (\bu_h \otimes \bu_h): \nabla \bv_h \dx -a(\bu_h, \bv_h) +  s_h^u(\bu_h, \bv_h)\\
    {}&+b(\bv-\bv_h, p_h)-
    \int_{\Omega} (\bu_h \otimes \bu_h): \nabla (\bv-\bv_h) \dx.
\end{align*}
The first term on the right-hand side of the formula for  $e_{\mathbf{u}_h} (\bv,t)$ has
already been estimated in \eqref{remark_equation}.~The remainder of the first line represents the consistency error due to mass lumping. It can again be estimated as in \cite[Thm. 3]{preprint}. The second line is the consistency error associated with the discretization of the nonlinear convective term.
By Lemmas \ref{lemma4} and \ref{lemma5}, this error tends to zero
as $h\to 0$. The term $b(\bv-\bv_h, p_h)$ measures the divergence error in the interpolant $\bv_h=I_h\bv$.
It can be estimated using the Cauchy--Schwarz inequality and the interpolation error estimate $\|\nabla\cdot(\bv-\bv_h)\|_{L^2(\Omega)}\le Ch|\bv|_{H^2(\Omega)}$. The second term in the third line can be estimated similarly.
Combining the above estimates for the individual components of $e_{\mathbf{u}_h} (\bv,t)$, we complete the proof of the theorem.
\end{proof}

% DK: Ich bin mir nicht sicher, ob Ritz und Stokes die O(h^2) Konsistenz des ersten Terms in (45) garantieren. Ausserdem müsste man diese Projektionen sauber definieren und Quellen mit Abschätzungen des Projektionsfehlers zitieren.
%\begin{remark}
%There are several options for choosing the projection from classical spaces to divergence-free spaces as used before equation \eqref{remark_equation}.
%The \(L^2\) projection onto the divergence-free space corresponds to the Darcy problem, where the error between a smooth function \(\mathbf{v} \in C^2(\Omega)\) and its projection \(\mathbf{v}_h\) is of order \(O(h^{k+1})\) in the \(L^2\) norm under the usual assumptions for finite element spaces of polynomial degree \(k\). The Ritz projection (or \(H^1\) projection) corresponds to the Stokes problem, and the error is \(O(h^k)\) in the \(H^1\) (or gradient) norm, as the Ritz projection minimizes the energy functional involving the gradients. Both projections enforce divergence-free conditions but with different norms and convergence rates tied to the associated problem. Here, both projections can be used and yield the desired results.
%\end{remark}

Finally, the assumptions of
the following theorem guarantee weak convergence to a DMV solution.

\begin{theorem}[Weak convergence]\label{eq:theorem_weak} 
Assume that Lemmas \ref{lemma1}--\ref{lemma5} are applicable on a time interval $[0,T]$ and let the initial data $\bu_{h,0}$ be admissible.
Consider a family $\{\mathbf{U}_h\}_{h \searrow 0} \equiv \{ \bu_h, p_h \}_{h \searrow 0}$ of uniformly bounded numerical solutions obtained with the scheme \eqref{weakform-stab}. Then there exists a subsequence of $\mathbf{U}_h$ (denoted again by $\mathbf{U}_h$) such that
$$
\mathbf{u}_h \to  \langle \mathcal U_{(\mathbf{x},t)}; \tilde{\bu} \rangle  \text{ weakly}-* \text{ in } L^{\infty} ((0,T); L^2(\Omega; \R^d))
$$
as $h\to 0$. The limit $\mathcal{U}_{(\mathbf{x},t)}$ is a DMV solution of the incompressible Euler equations corresponding to the initial data $\bu_0$ and concentration defects
\begin{equation*}
%  \begin{aligned}
\mathcal{D}_{cd}(t)= \overline{\eta(\bu)}(t)- \frac{1}{2}\langle \mathcal U_{(\mathbf x,t)}; |  \tilde{\bu}|^2 \rangle, \qquad %    \\
\mathfrak{M}_{cd} (t)= \overline{\bu \otimes \bu} -  \langle \mathcal U_{(\mathbf x,t)}; \tilde{\bu} \otimes  \tilde{\bu} \rangle,
%\end{aligned}
\end{equation*}
where \(\overline{\eta(\bu)}\) is the weak-$*$ limit of \(\eta(\bu_h)=\frac{1}{2} |\bu_h|^2 \) as $ h \to 0 $ in \(L^\infty(0, T; \mathcal{M}^+(\Omega))\) and \(\bu \otimes \bu\) is the weak-$*$ limit of \(\{\bu_h \otimes \bu_h\}_{h }\) in \(L^\infty(0, T; \mathcal{M}(\Omega; \mathbb{R}^{d \times d}))\).
\end{theorem}
\begin{proof}
The result of Theorem~\ref{eq:theorem_weak} does not rely on any details of the discretization procedure. Instead, the proof technique that we use in this work exploits the inherent properties of our structure-preserving numerical scheme. In view of the consistency result that they imply, we can closely follow the convergence proof outlined in \cite{arun2024asymptoticpreservingfinitevolume}. Its central idea involves leveraging the uniform boundedness of various quantities, applying the Fundamental Theorem of Young measures, and using the consistency estimates provided by Theorem~\ref{th_consistency}.
Strong convergence as well as $\mathcal{K}$ convergence can then be shown similarly to the compressible/hyperbolic case. For technical details, we refer the interested reader to \cite{arun2024asymptoticpreservingfinitevolume, preprint}.
\end{proof}

\section{Numerical results}
\label{sec:examples}

We apply our locally energy-stable finite element method \eqref{galerkin-stab} with $d_{ij}^u=0$ to standard 2D test problems in this section.
Following Becker and Hansbo~\cite{becker}, the pressure stabilization parameter $\omega$ is set to $0.5$ in all experiments.
Our implementation invokes mesh generation techniques from the open-source packages FESTUNG~\cite{festung,reuter2021} and MFEM~\cite{anderson2021,mfem}.
In a postprocessing step, periodicity is enforced by manipulating the grid data as required.
This task is accomplished by identifying nodes (and also edges, depending on the implementation/discretization) that represent the same entity on periodic domains.
The numerical approximations are visualized with the open-source software GLVis~\cite{glvis}.
Initial conditions for the velocity vector are generated via consistent $L^2(\Omega)$ projections (by default) or lumped $L^2(\Omega)$ projections if the initial conditions are nonsmooth (Gresho vortex, Section \ref{sec:gresho}).

\subsection{Taylor--Green vortex}\label{sec:tg}

First, we assess the accuracy of the method under investigation by running the Taylor--Green vortex test in $\Omega=(0,1)^2$.
The smooth exact solution to this incompressible flow problem reads~\cite{schroeder2019}
\begin{align}\label{eq:tg}
\mathbf u(x,y,t) ={}& \begin{bmatrix}
\sin(2\pi x)\sin(2\pi y) \\ \cos(2\pi x)\cos(2\pi y)
\end{bmatrix}
\exp(-8\pi^2\nu t), \\
p(x,y,t) ={}& \frac12 [1 - \sin^2(2\pi x) - \cos^2(2\pi y)]\exp(-16\pi^2\nu t).\notag
\end{align}
Note that $\int_\Omega p(x,y,t)\mathrm dx \mathrm dy = 0$ for all $t\ge 0$.
Following~\cite[Sec.~4.3]{schroeder2019}, we set $\nu=10^{-5}$ and run simulations up to the end time $t=1$, at which we compute the $L^2(\Omega)$ errors $e_\mathbf{u}^M$, $e_p^M$ in the velocity $\mathbf u$ and pressure $p$.
The superscript $M\in \{C,L\}$ indicates whether calculations are performed with the consistent or lumped mass matrix.
Using a hierarchy of uniform Friedrichs--Keller triangulations, unstructured Delaunay meshes and uniform quadrilateral grids, we calculate the experimental orders of convergence~(EOC).
The ratio $\Delta t/h$, where $h=\max_e$diam$(K^e)$, is kept fixed in each case with values of $\Delta t/h=\sqrt2/4\approx0.3536$ for both types of structured grids and $\Delta t/h=5/16 = 0.3125$ for computations on unstructured meshes.
These choices amount to a total of 512 time steps on the finest mesh of each type.
The results of our grid convergence study are presented in Tables~\ref{tab:tgv1}--\ref{tab:tgv3}.

\begin{table}[ht!]
\centering
\begin{tabular}{c||cc|cc||cc|cc}
$\sqrt2/h$ & $e_\mathbf{u}^C$ & EOC & $e_p^C$ & EOC & $e_\mathbf{u}^L$ & EOC & $e_p^L$ & EOC \\
\hline
16  & 8.01E-02 &      & 6.59E-03 &      & 6.75E-02 &      & 5.74E-03 &      \\
32  & 1.90E-02 & 2.08 & 1.37E-03 & 2.26 & 1.82E-02 & 1.89 & 1.35E-03 & 2.09 \\
64  & 4.77E-03 & 2.00 & 3.55E-04 & 1.95 & 4.71E-03 & 1.95 & 3.53E-04 & 1.93 \\
128 & 1.20E-03 & 1.99 & 9.03E-05 & 1.98 & 1.20E-03 & 1.98 & 9.03E-05 & 1.97 \\
256 & 3.02E-04 & 1.99 & 2.28E-05 & 1.99 & 3.02E-04 & 1.99 & 2.28E-05 & 1.98 \\
\hline
average && 2.02 && 2.05 && 1.95 && 1.99
\end{tabular}
\caption{L$^2(\Omega)$ convergence for the Taylor--Green vortex on Friedrichs--Keller triangulations.}\label{tab:tgv1}
\end{table}

\begin{table}[ht!]
\centering
\begin{tabular}{c||cc|cc||cc|cc}
$1/h$ & $e_\mathbf{u}^C$ & EOC & $e_p^C$ & EOC & $e_\mathbf{u}^L$ & EOC & $e_p^L$ & EOC \\
\hline
10  & 2.50E-01 &      & 1.20E-01 &      & 6.24E-02 &      & 2.98E-02 &      \\
20  & 7.58E-02 & 1.72 & 2.93E-02 & 2.03 & 2.16E-02 & 1.53 & 9.07E-03 & 1.71 \\
40  & 2.32E-02 & 1.71 & 8.82E-03 & 1.73 & 7.40E-03 & 1.55 & 3.15E-03 & 1.53 \\
80  & 6.92E-03 & 1.75 & 2.09E-03 & 2.07 & 2.50E-03 & 1.56 & 9.93E-04 & 1.66 \\
160 & 1.18E-03 & 2.55 & 4.68E-04 & 2.16 & 8.13E-04 & 1.62 & 3.06E-04 & 1.70 \\
\hline
average && 1.93 && 2.00 && 1.57 && 1.65
\end{tabular}
\caption{L$^2(\Omega)$ convergence for the Taylor--Green vortex on unstructured Delaunay meshes.}
\end{table}

\begin{table}[ht!]
\centering
\begin{tabular}{c||cc|cc||cc|cc}
$\sqrt2/h$ & $e_\mathbf{u}^C$ & EOC & $e_p^C$ & EOC & $e_\mathbf{u}^L$ & EOC & $e_p^L$ & EOC \\
\hline
16  & 8.83E-03 &      & 4.59E-03 &      & 7.38E-03 &      & 4.22E-03 &      \\
32  & 2.11E-03 & 2.07 & 1.01E-03 & 2.19 & 2.00E-03 & 1.88 & 9.77E-04 & 2.11 \\
64  & 5.23E-04 & 2.01 & 2.44E-04 & 2.05 & 5.16E-04 & 1.96 & 2.42E-04 & 2.01 \\
128 & 1.31E-04 & 2.00 & 6.07E-05 & 2.00 & 1.30E-04 & 1.99 & 6.07E-05 & 2.00 \\
256 & 3.27E-05 & 2.00 & 1.53E-05 & 1.99 & 3.26E-05 & 2.00 & 1.53E-05 & 1.99 \\
\hline
average && 2.02 && 2.08 && 1.96 && 2.03
\end{tabular}
\caption{L$^2(\Omega)$ convergence for the Taylor--Green vortex on uniform quadrilateral grids.}\label{tab:tgv3}
\end{table}

We observe that both the lumped and the consistent mass matrix versions of our scheme converge with optimal orders of accuracy in every variable.
To further quantify the dissipative properties of the new discretizations, we track the temporal evolutions of kinetic energies and compare them with the exact energy.
According to \eqref{eq:tg}, the latter decays exponentially in time at the rate of $-16\pi^2\nu$.
Since the kinetic energies of discrete initial data are generally mesh dependent, we normalize the data sets by their respective initial values.
The results are presented in Fig.~\ref{fig:tg-ene}.
Note that the kinetic energies of lumped and consistent mass matrix approximations are computed differently, namely as follows:
\begin{align*}
\eta_L(\mathbf u_h) = \frac12 (\mathbf u_h,\mathbf u_h)_{L^2(\Omega)},\qquad
\eta_C(\mathbf u_h) = \frac12 (\mathbf u_h,\mathbf u_h)_{L_h^2(\Omega)}.
\end{align*}

\begin{figure}[ht!]
\centering
\begin{subfigure}[b]{0.32\textwidth}
\caption{Friedrichs--Keller}
\includegraphics[width=\textwidth,trim=75 0 50 0,clip]{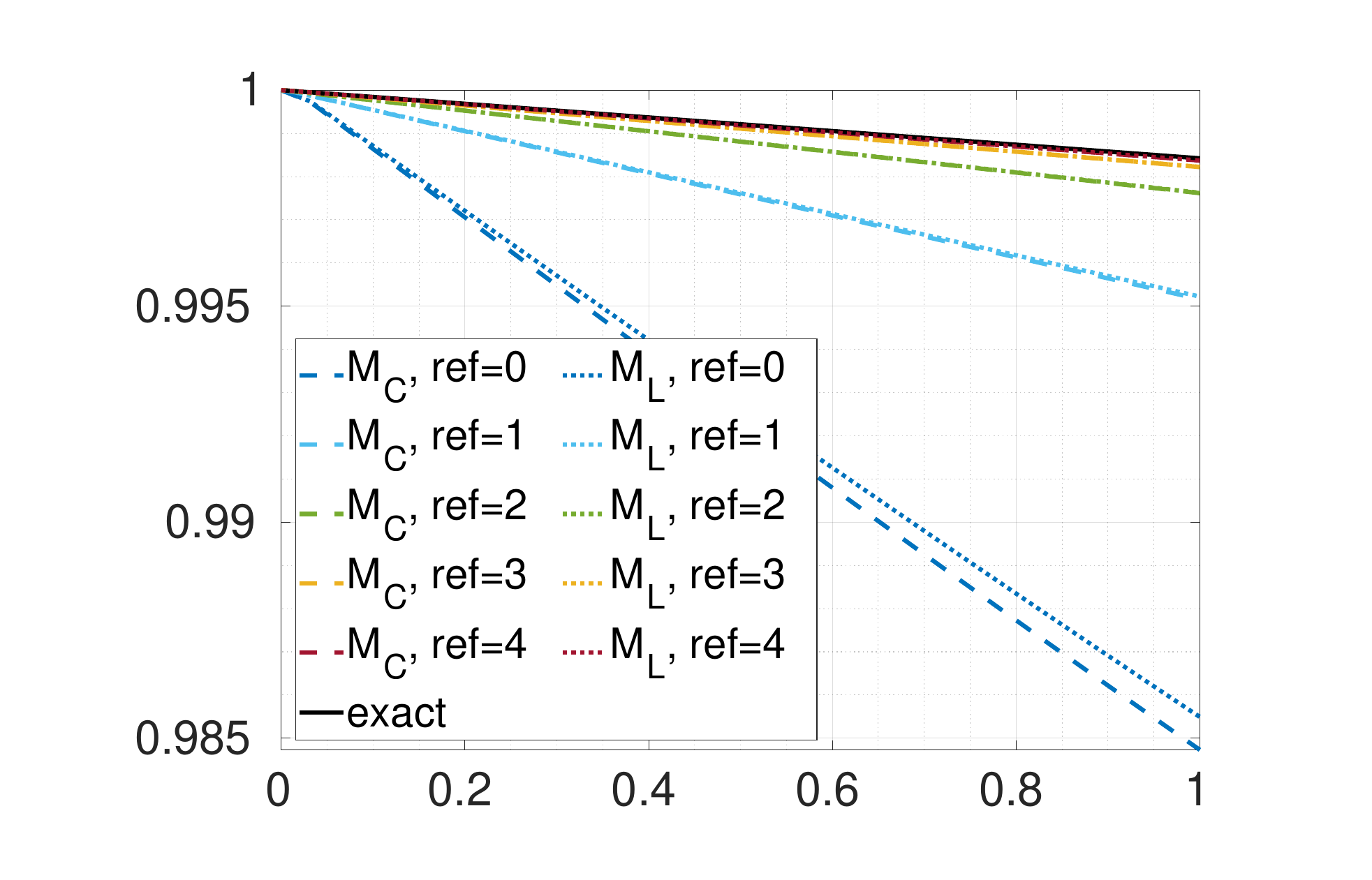}
\end{subfigure}
\begin{subfigure}[b]{0.32\textwidth}
\caption{Unstructured Delaunay}
\includegraphics[width=\textwidth,trim=75 0 50 0,clip]{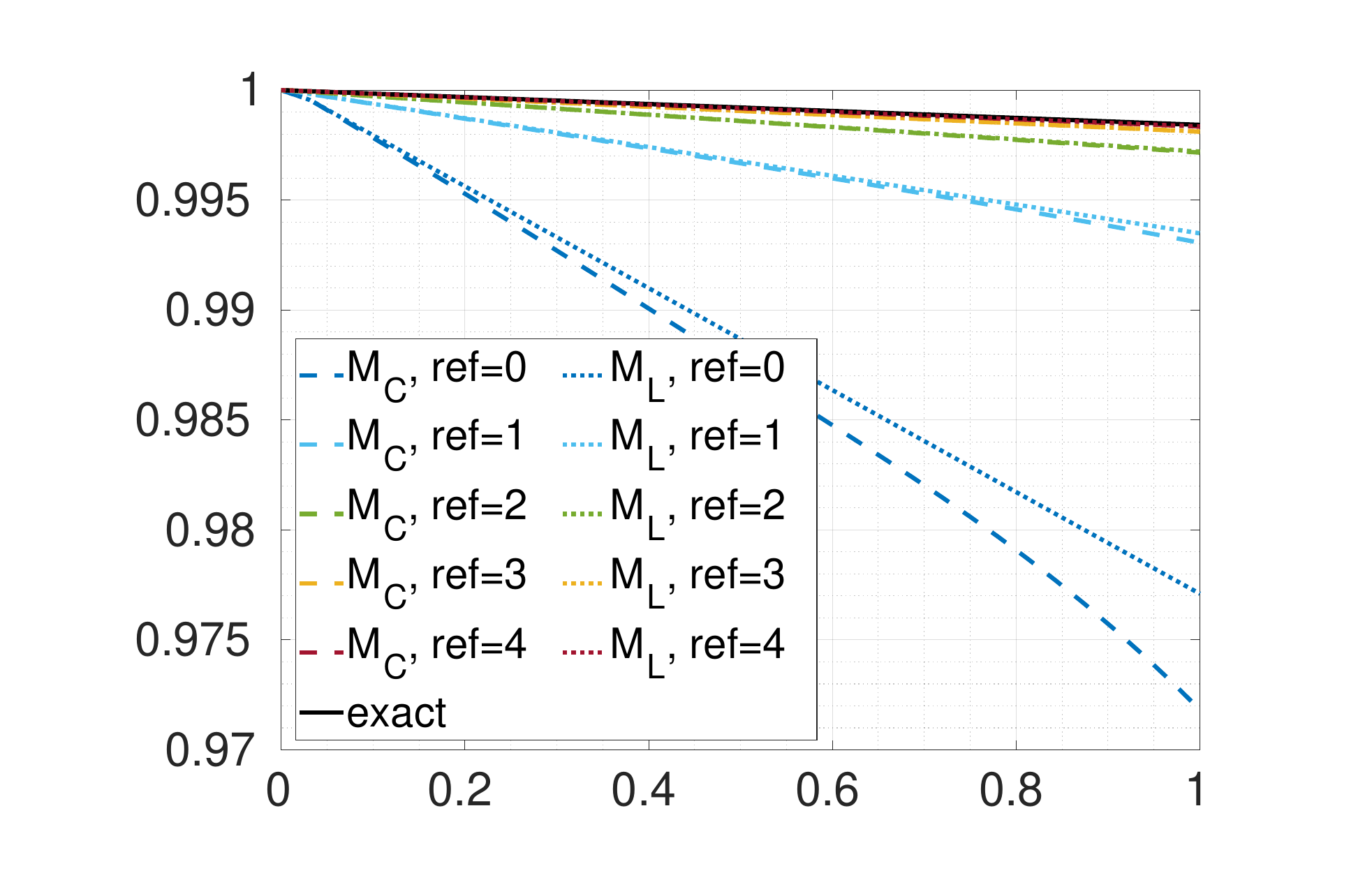}
\end{subfigure}
\begin{subfigure}[b]{0.32\textwidth}
\caption{Uniform quadrilateral}
\includegraphics[width=\textwidth,trim=75 0 50 0,clip]{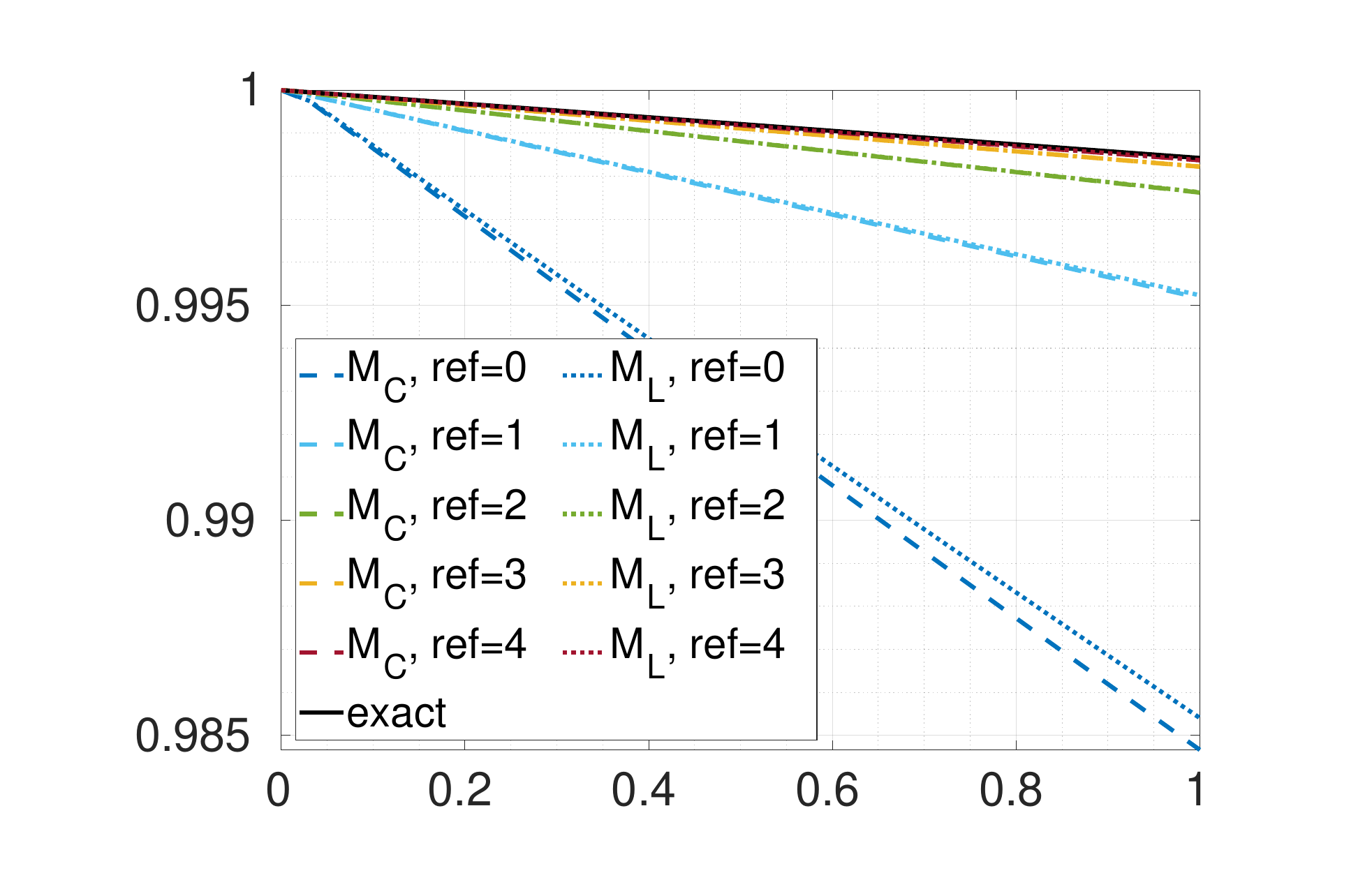}
\end{subfigure}
\caption{Energy evolutions for the Taylor--Green vortex, exact profile and numerical results obtained on refinement levels \texttt{ref} for three types of meshes.}\label{fig:tg-ene}
\end{figure}

Generally, the lumped-mass solutions seem to be less diffusive than their consistent-mass counterparts.
This observation is illustrated by the error behavior in Tables~\ref{tab:tgv1}--\ref{tab:tgv3}.
Upon spatial and temporal refinement, each profile converges to that for the exact energy evolution.
While the analysis presented in Section \ref{sec:strong} guarantees that the total energy is nonincreasing for the semi-discrete scheme and exact time integration, our numerical results indicate that the global energy stability property carries over to the fully discrete setting.
Indeed, all profiles obtained in this study are monotonically decreasing.

Next, we study the long-time solution behavior by running the simulations up to an end time that is 50 times larger than before.
Again, we use lumped and consistent mass matrices, as well as the three types of meshes.
In Figs.~\ref{fig:tg-v}--\ref{fig:tg-p}, we present snapshots of the velocity magnitude and pressure obtained with the lumped-mass version of our scheme on the finest level of a structured quadrilateral grid.
The numerical solutions obtained with any of the other five computational setups are qualitatively similar to the displayed approximations.
Kinetic energy profiles for all runs are shown in Fig.~\ref{fig:tg-long-ene}.

\begin{figure}[ht!]
\centering
\begin{subfigure}[b]{0.24\textwidth}
\caption{$t=1,|\mathbf u|\le 1.00$}
\includegraphics[width=\textwidth]{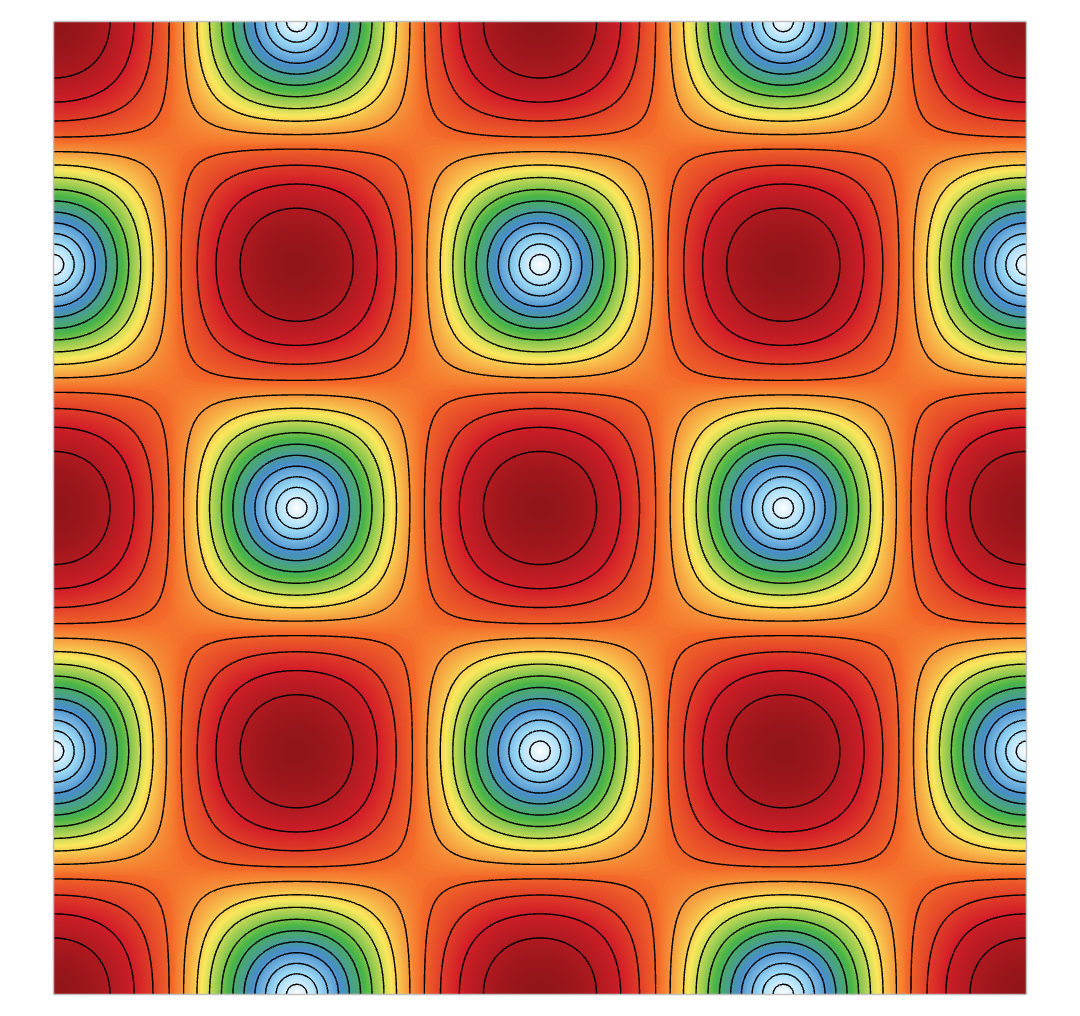}
\end{subfigure}
\begin{subfigure}[b]{0.24\textwidth}
\caption{$t=35,|\mathbf u|\le 0.973$}
\includegraphics[width=\textwidth]{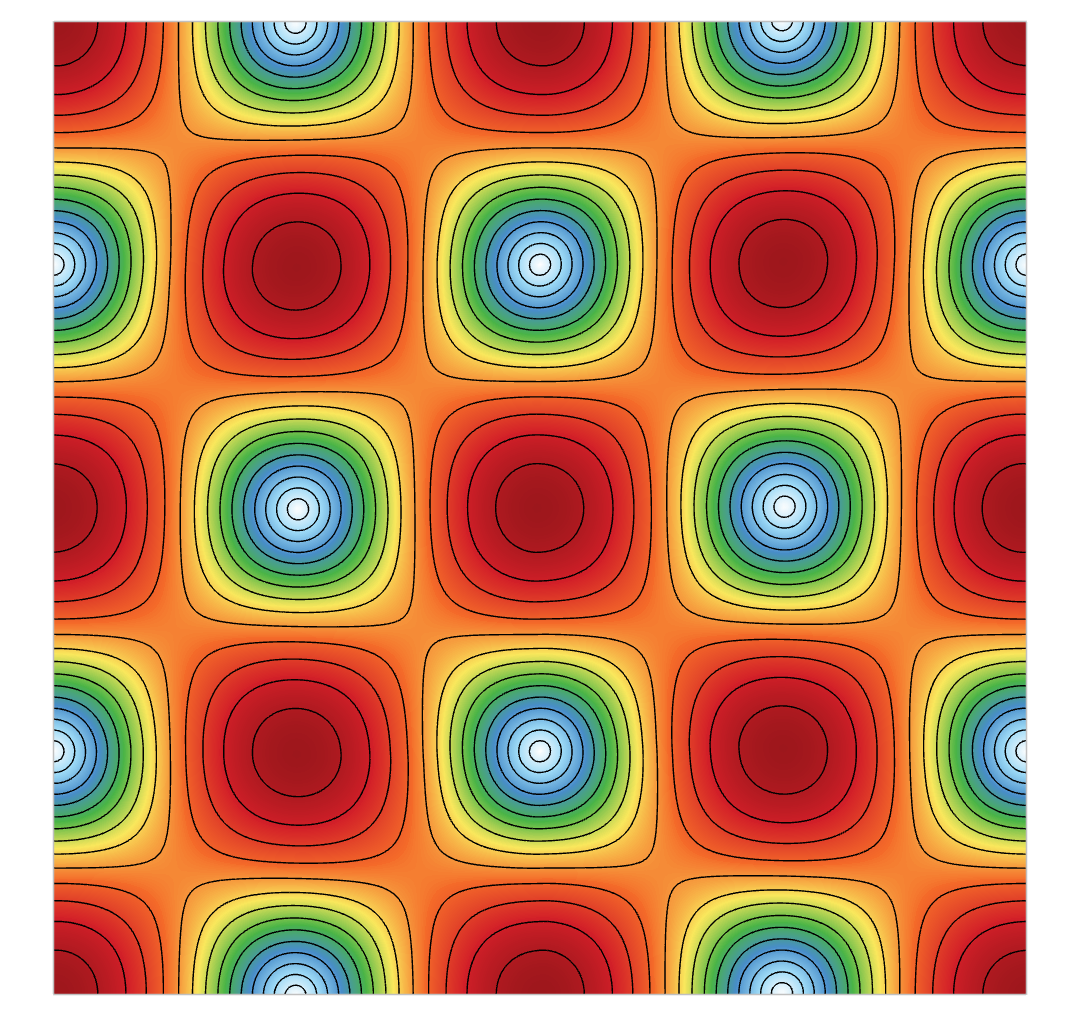}
\end{subfigure}
\begin{subfigure}[b]{0.24\textwidth}
\caption{$t=40,|\mathbf u|\le 0.987$}
\includegraphics[width=\textwidth]{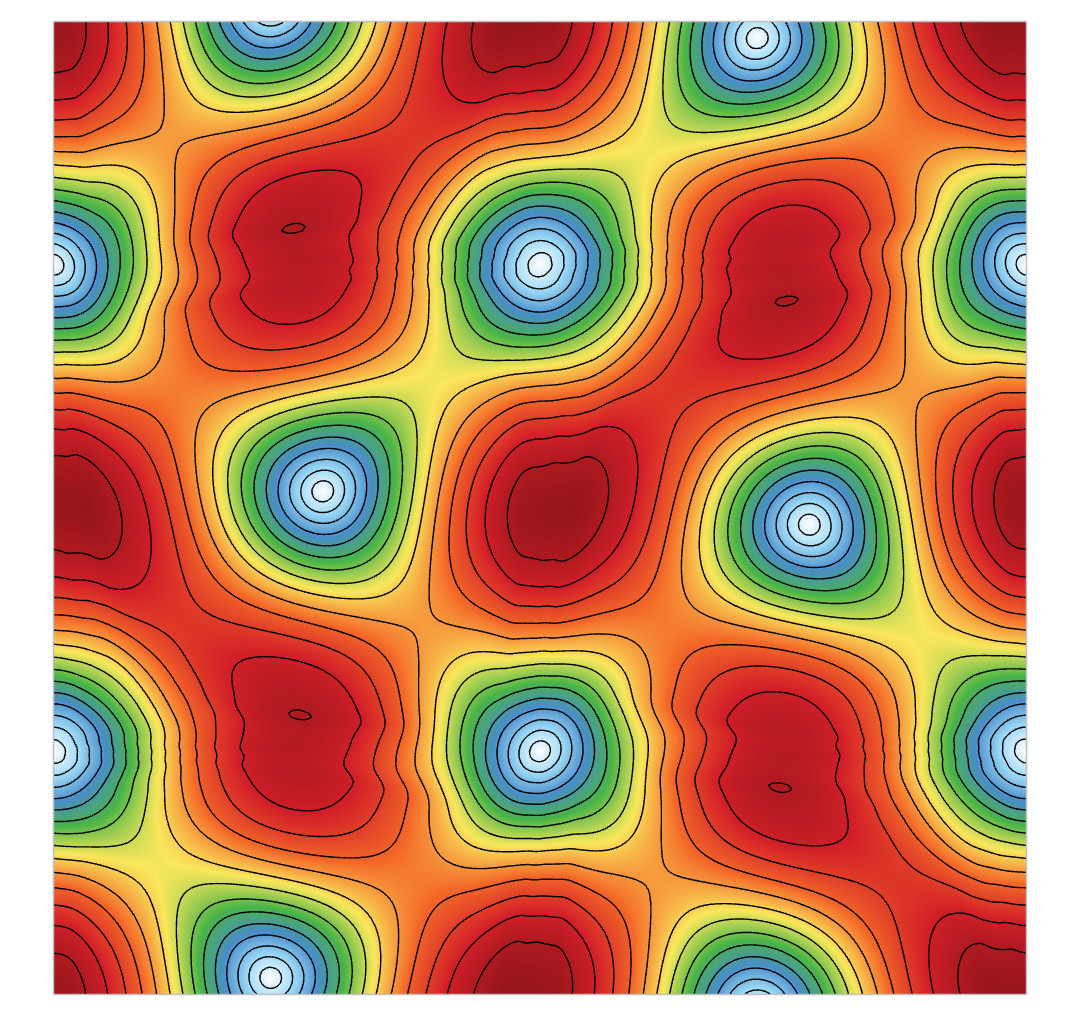}
\end{subfigure}
\begin{subfigure}[b]{0.24\textwidth}
\caption{$t=41,|\mathbf u|\le 0.988$}
\includegraphics[width=\textwidth]{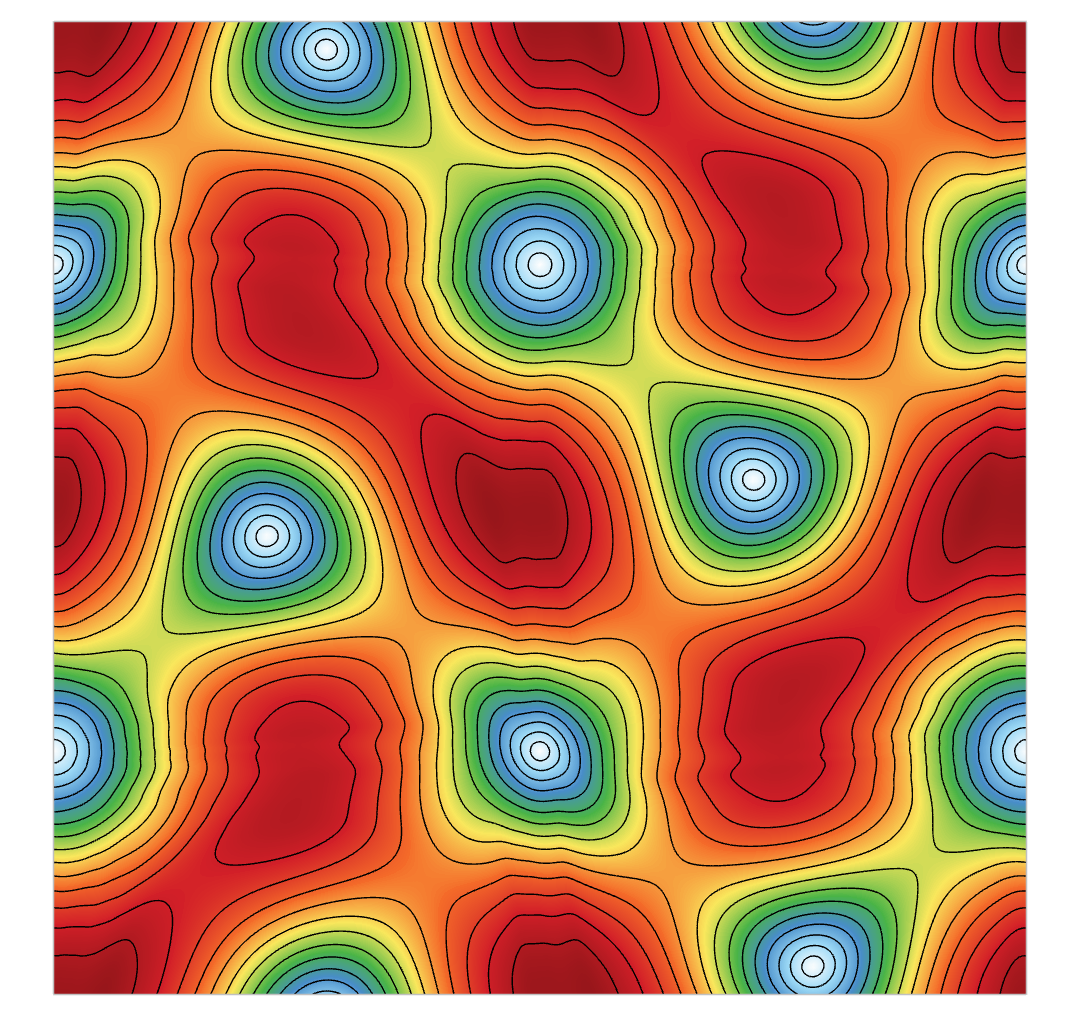}
\end{subfigure}
\begin{subfigure}[b]{0.24\textwidth}
\caption{$t=42|\mathbf u|\le 0.970$}
\includegraphics[width=\textwidth]{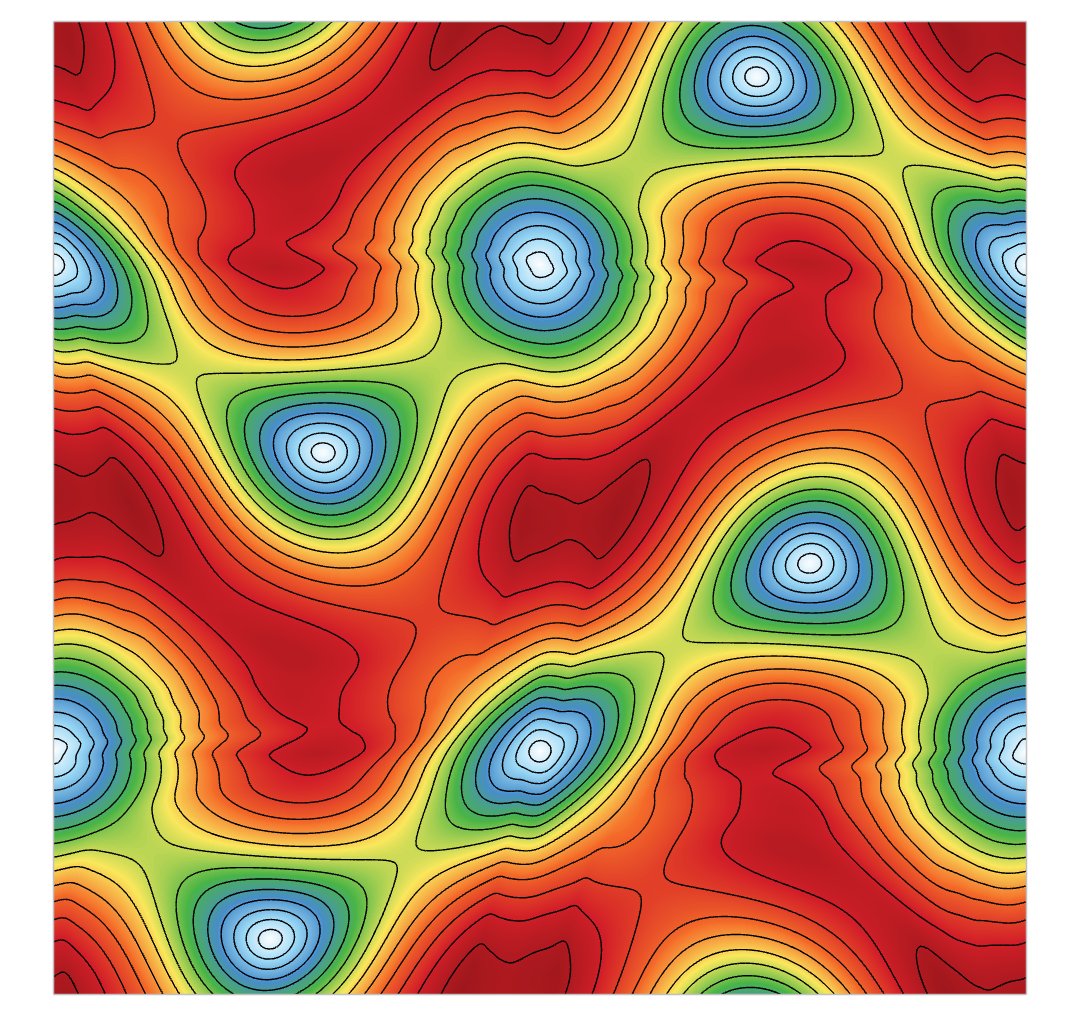}
\end{subfigure}
\begin{subfigure}[b]{0.24\textwidth}
\caption{$t=43,|\mathbf u|\le 1.11$}
\includegraphics[width=\textwidth]{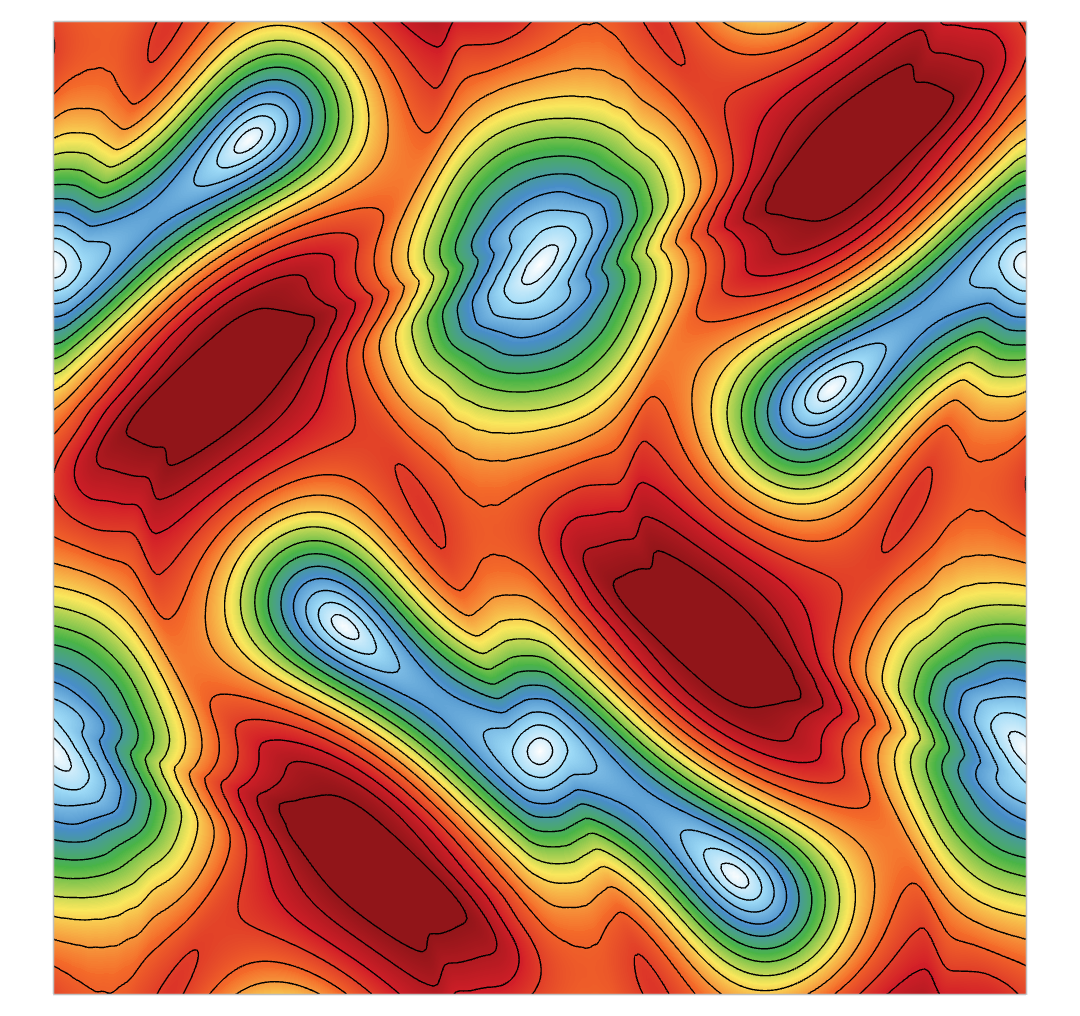}
\end{subfigure}
\begin{subfigure}[b]{0.24\textwidth}
\caption{$t=44,|\mathbf u|\le 1.06$}
\includegraphics[width=\textwidth]{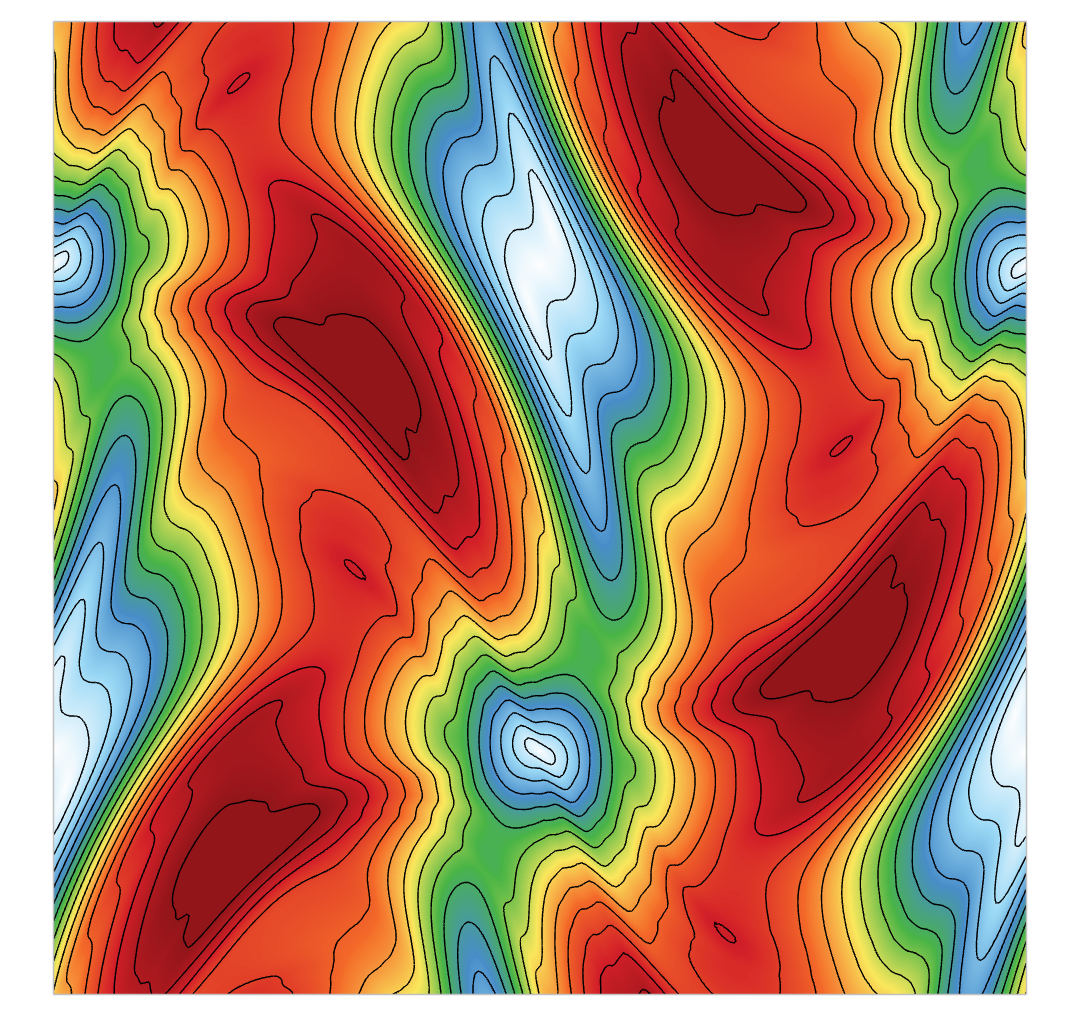}
\end{subfigure}
\begin{subfigure}[b]{0.24\textwidth}
\caption{$t=45,|\mathbf u|\le 1.17$}
\includegraphics[width=\textwidth]{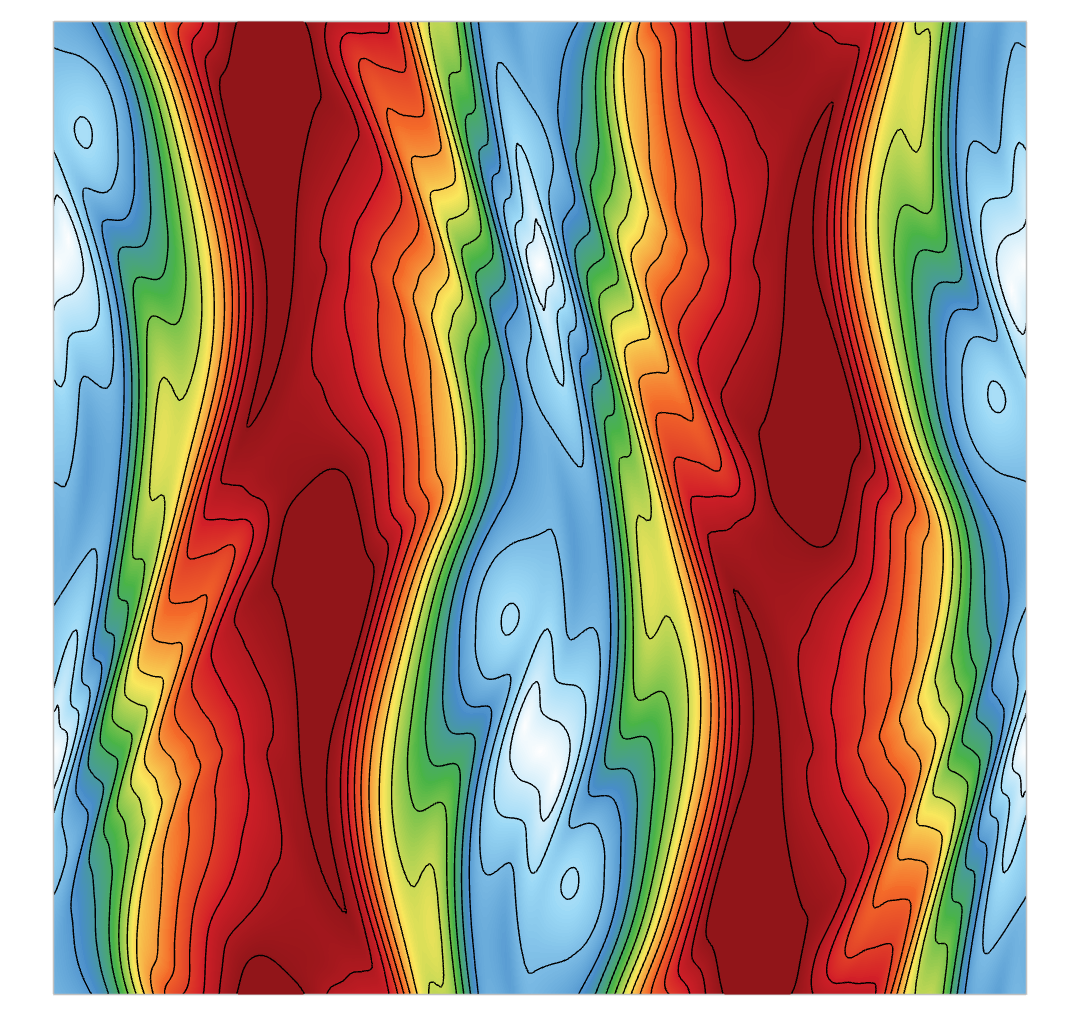}
\end{subfigure}
\caption{Snapshots of finite element approximations to $|\mathbf u|$ for the Taylor--Green vortex obtained on a uniform quadrilateral mesh with $256\times 256$ elements using the lumped mass matrix and the time step $\Delta t = 1/512$.}\label{fig:tg-v}
\end{figure}

\begin{figure}[ht!]
\centering
\begin{subfigure}[b]{0.24\textwidth}
\caption{$t=1,p\in[-\frac12,\frac12]$}
\includegraphics[width=\textwidth]{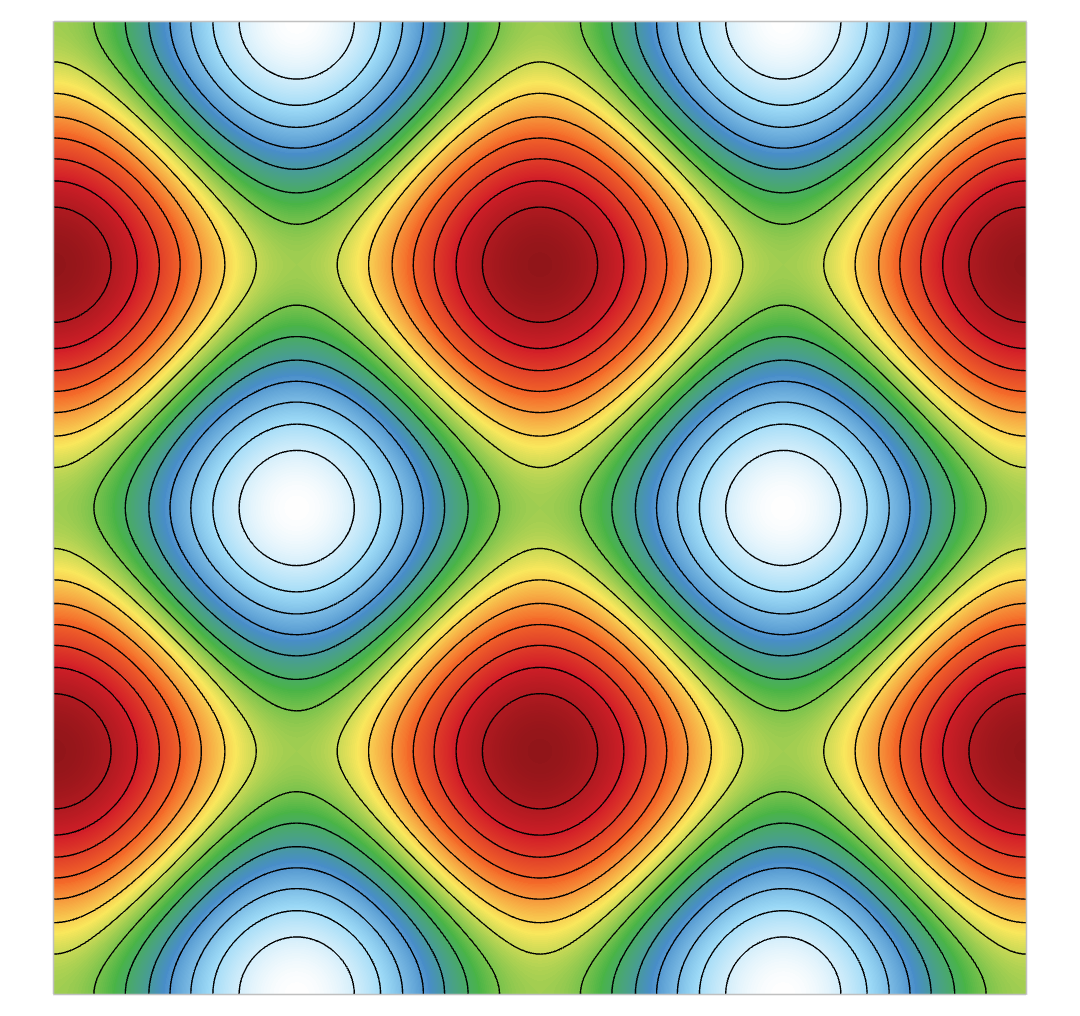}
\end{subfigure}
\begin{subfigure}[b]{0.24\textwidth}
\caption{$t=35,p\in[-\frac12,\frac12]$}
\includegraphics[width=\textwidth]{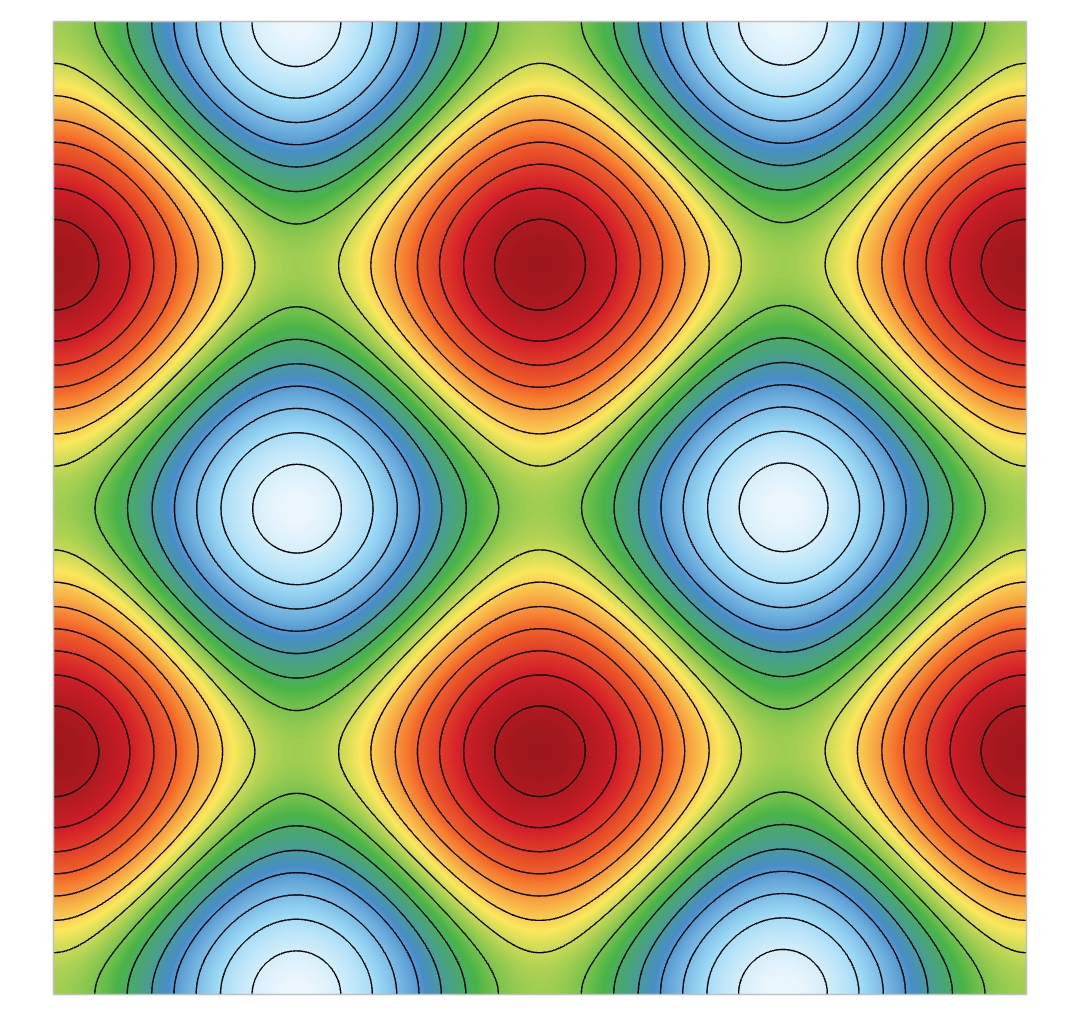}
\end{subfigure}
\begin{subfigure}[b]{0.24\textwidth}
\caption{$t=40,p\in[-\frac12,\frac12]$}
\includegraphics[width=\textwidth]{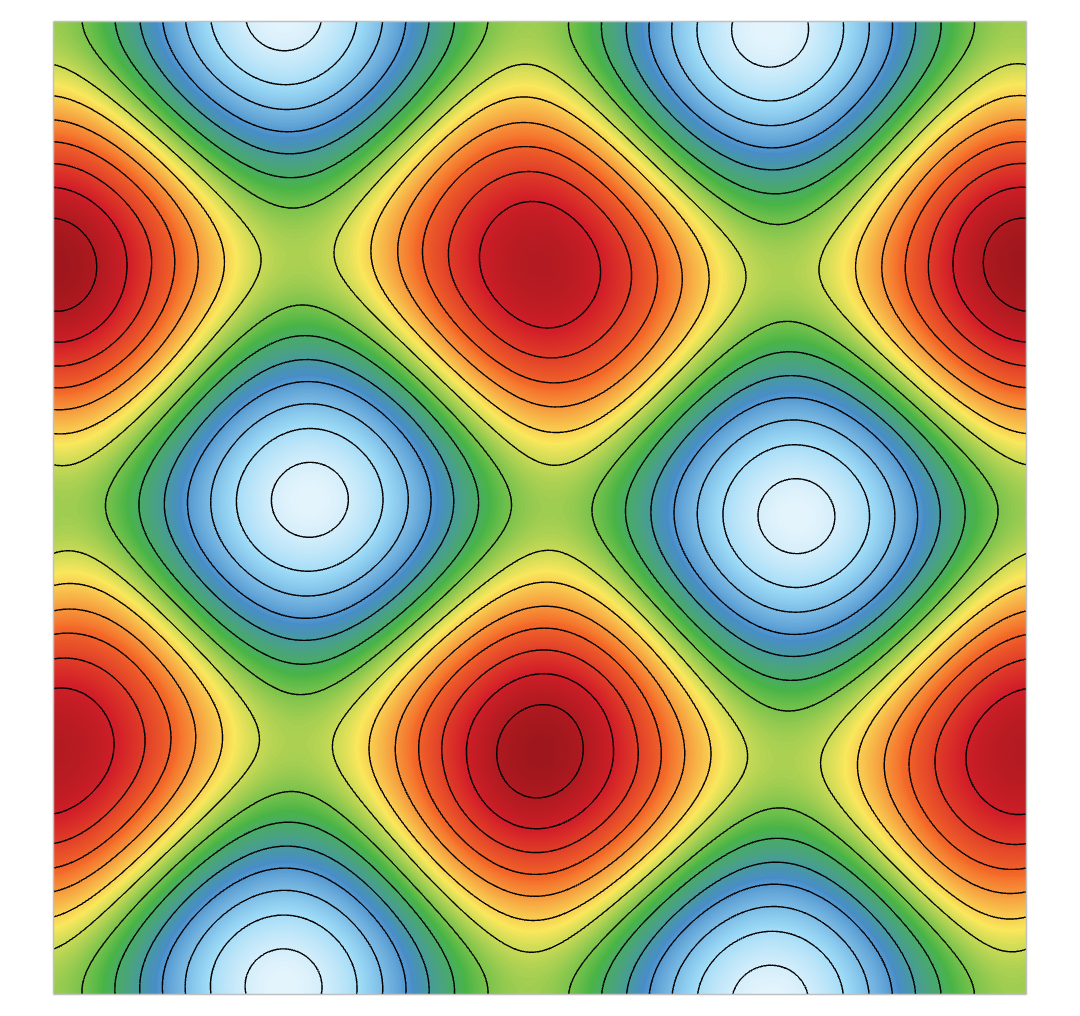}
\end{subfigure}
\begin{subfigure}[b]{0.24\textwidth}
\caption{$t=41,p\in[-\frac12,\frac12]$}
\includegraphics[width=\textwidth]{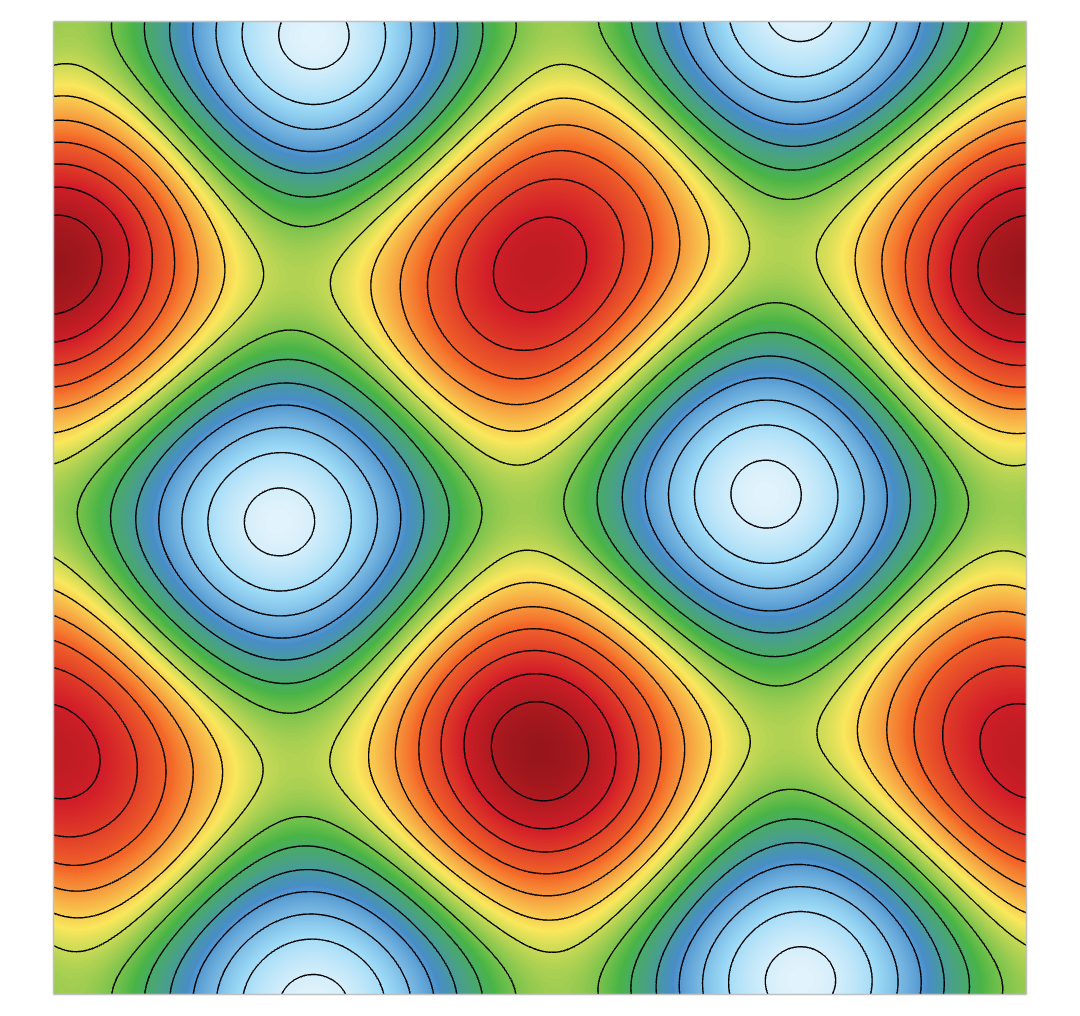}
\end{subfigure}
\begin{subfigure}[b]{0.24\textwidth}
\caption{$t=42,p\in[-\frac12,\frac12]$}
\includegraphics[width=\textwidth]{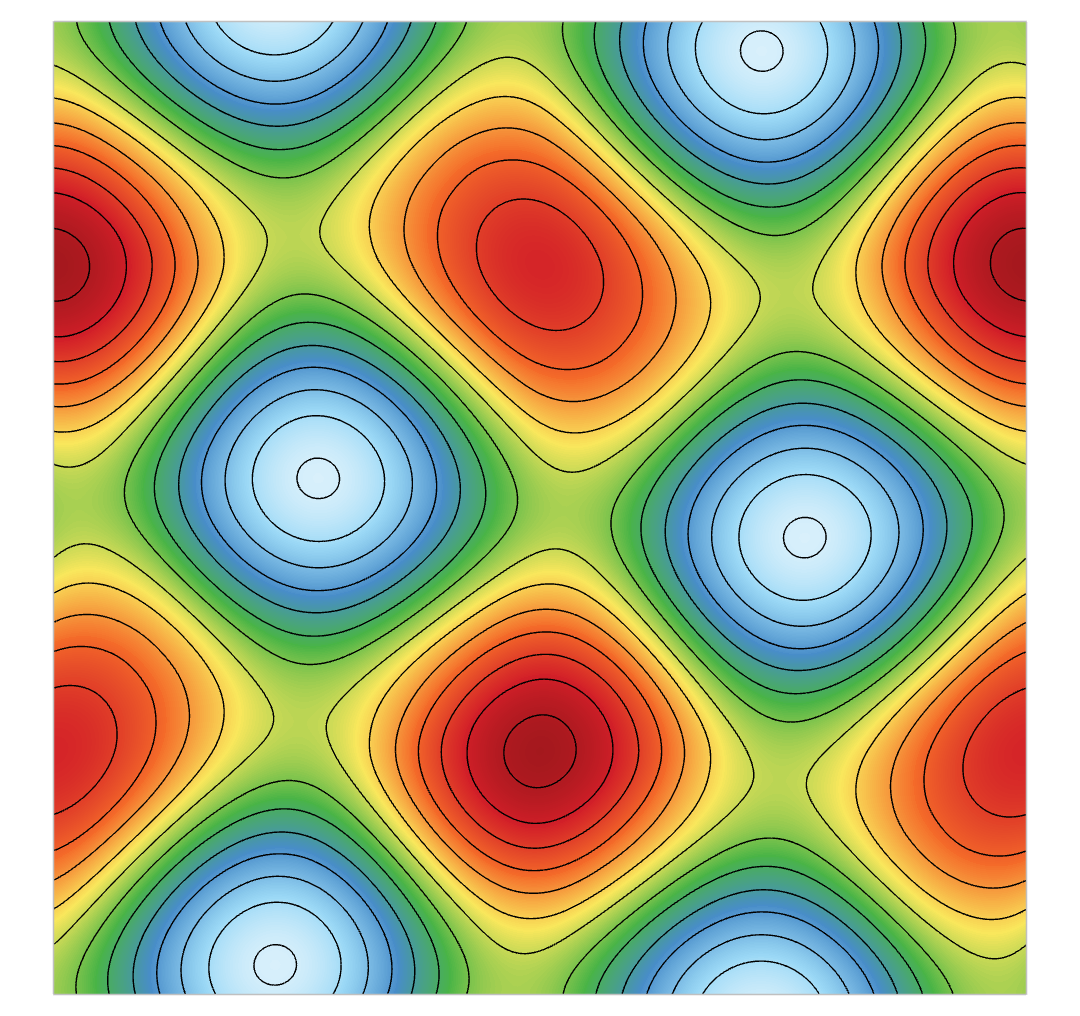}
\end{subfigure}
\begin{subfigure}[b]{0.24\textwidth}
\caption{$t=43,p\in[-\frac25,\frac25]$}
\includegraphics[width=\textwidth]{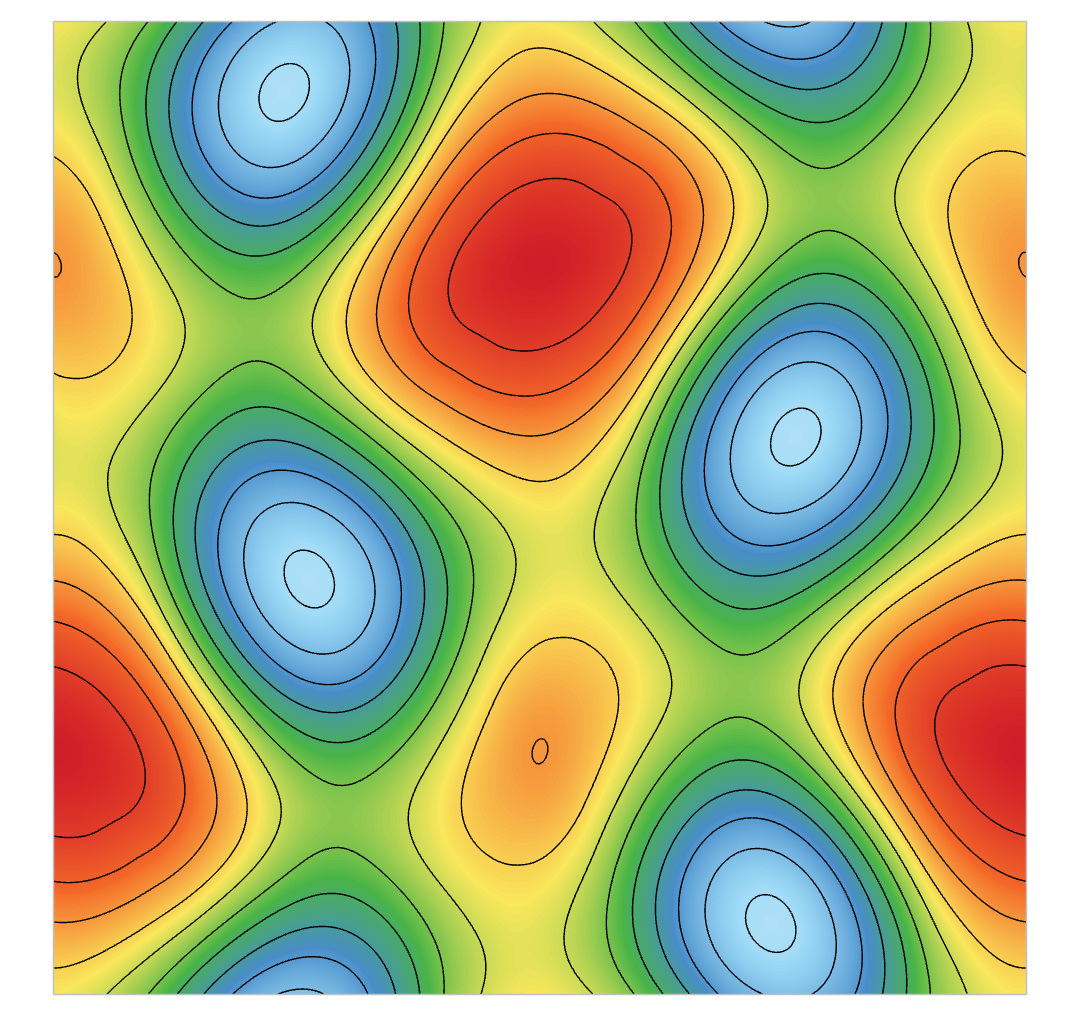}
\end{subfigure}
\begin{subfigure}[b]{0.24\textwidth}
\caption{$t=44,p\in[-\frac15,\frac25]$}
\includegraphics[width=\textwidth]{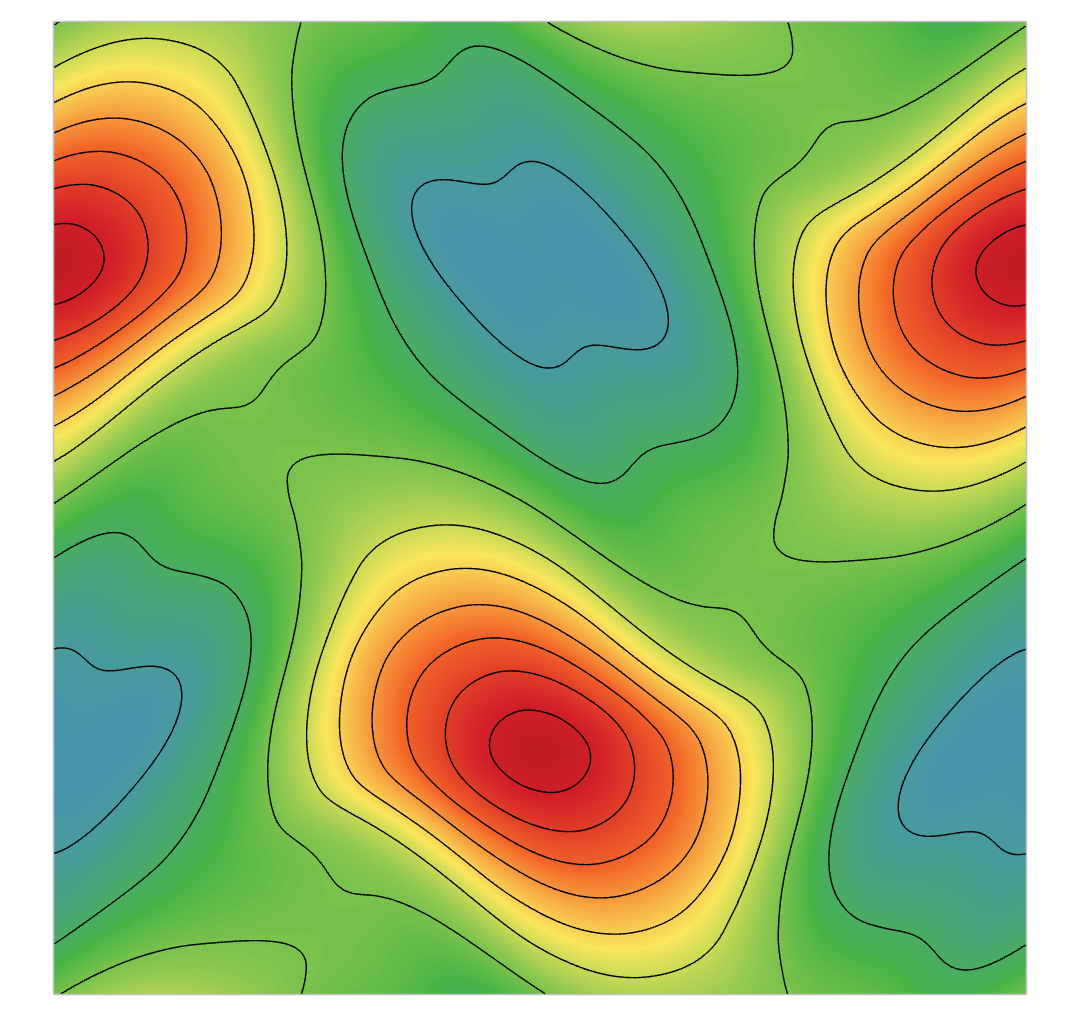}
\end{subfigure}
\begin{subfigure}[b]{0.24\textwidth}
\caption{$t=45,p\in[-\frac15,\frac3{10}]$}
\includegraphics[width=\textwidth]{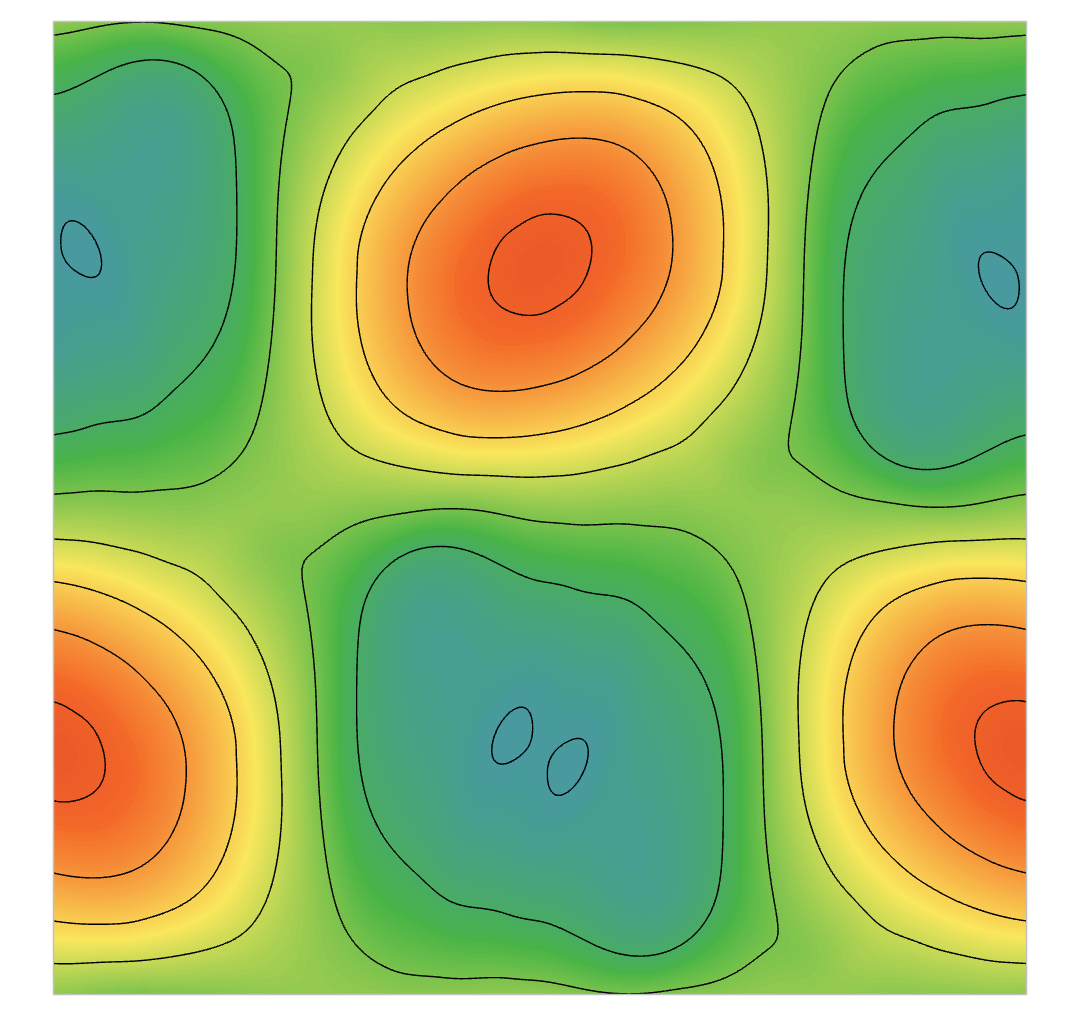}
\end{subfigure}
\caption{Snapshots of finite element approximations to $p$ for the Taylor--Green vortex obtained on a uniform quadrilateral mesh with $256\times 256$ elements using the lumped mass matrix and $\Delta t = 1/512$.}\label{fig:tg-p}
\end{figure}

\begin{figure}[ht!]
\centering
\includegraphics[width=0.5\textwidth]{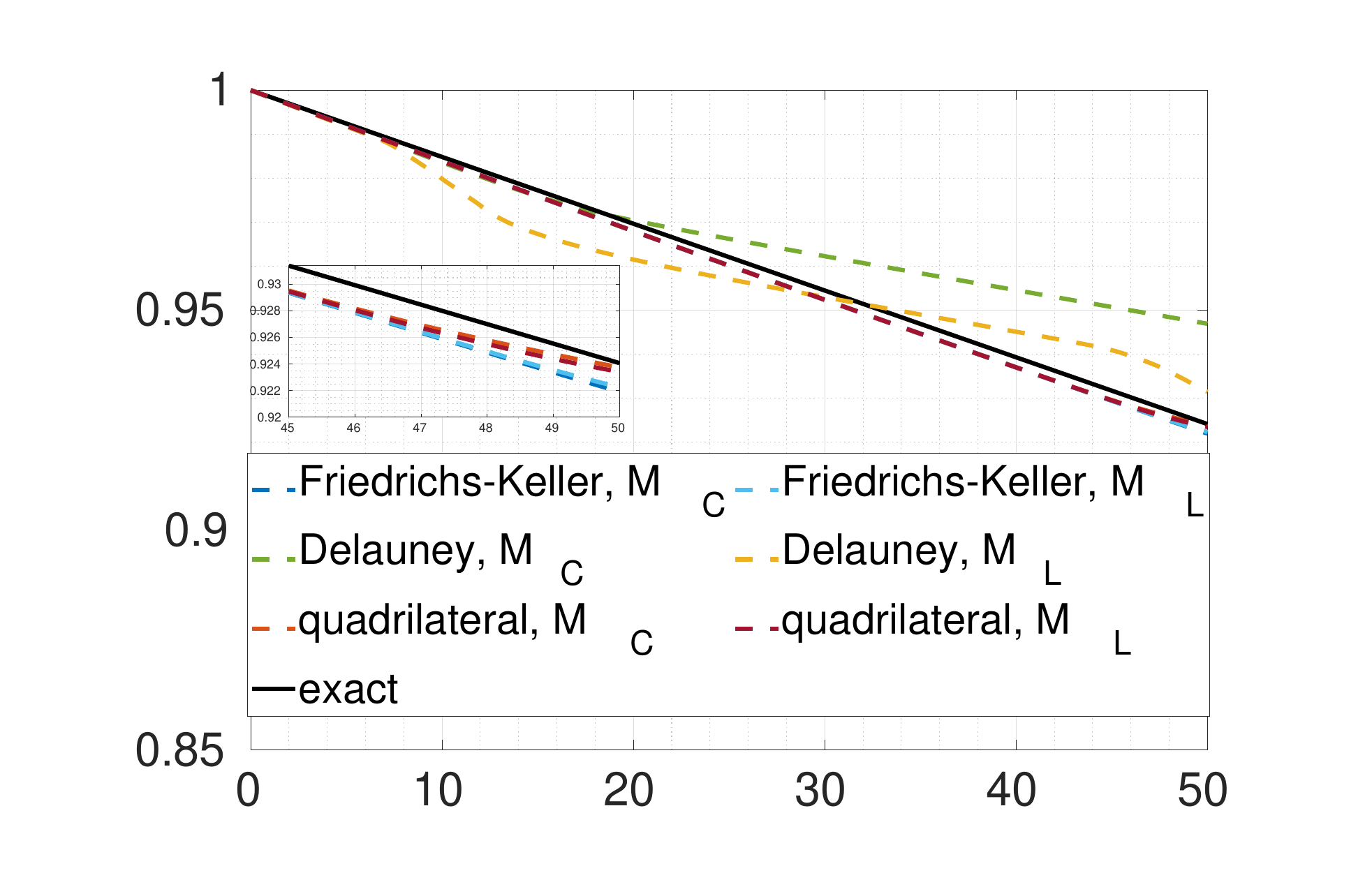}
\caption{Long-term energy evolutions for the Taylor--Green vortex on finest meshes.}\label{fig:tg-long-ene}
\end{figure}

In Figs.~\ref{fig:tg-v}--\ref{fig:tg-p}, we observe the so-called inverse cascade.
The originally coherent vortical structures of the Taylor cells, as seen in Figs. (a) and (b), undergo a self-organization after a startup phase (here after $t=35$).
The resulting flow patterns are depicted in Figs. (c)--(h); see also~\cite{lube2020}.
A mathematical discussion of such self-organization in 2D turbulent incompressible flows can be found in~\cite{van-groesen1988}.

\subsection{Gresho vortex}
\label{sec:gresho}

Next, we solve a test problem with a nonsmooth solution, known as the Gresho vortex~\cite{gresho,gresho1990}.
The physical viscosity $\nu$ is zero in this experiment.
That is, we are now considering the incompressible Euler limit \eqref{eq_incompressible_Euler} of \eqref{eq_incompressible}.
The exact solution coincides with the initial data, which is defined as follows:
\begin{align*}
r(x,y) ={}& \sqrt{x^2+y^2}, \qquad
\mathbf u(x,y) = (-y,x) \begin{cases}
5 & \text{if } r(x,y) \le \frac15, \\
\frac2{r(x,y)} - 5 & \text{if } \frac15 \le r(x,y) \le \frac25, \\
0 & \text{otherwise,}
\end{cases} \\
p(x,y) ={}& \begin{cases}
5 + \frac{25}2 r(x,y)^2 & \text{if } r(x,y) \le \frac15, \\
9 - 4\log\left(\frac15\right) + \frac{25}2 r(x,y)^2 - 20r(x,y) + 4\log(r(x,y)) & \text{if } \frac15 \le r(x,y) \le \frac25, \\
3 + 4\log(2) & \text{otherwise.}
\end{cases}
\end{align*}
The purpose of this popular 2D test is to check how well the method under investigation can preserve the exact steady state.
Note that the velocity is of class $C^0$, while the pressure is of class $C^1$.
Therefore, second-order convergence for the pressure is possible, while, according to~\cite{gresho}, the empirical convergence rate for the velocity is around 1.4.
We also remark that we need to normalize $p$ in the above definition to satisfy the zero mean condition.
To this end, we subtract $p_0 = \int_\Omega p(x,y) \mathrm dx\mathrm dy$ from $p$.
The value of~$p_0$ can be accurately evaluated by subdividing $\Omega$ and changing to polar coordinates.

We run the same experiments as for the Taylor--Green vortex at the beginning of Section~\ref{sec:tg}.
The computational domain is chosen to be $\Omega=(-0.5,0.5)^2$, which yields $p_0\approx 5.688812918144054$.
Simulations are terminated at the pseudo-time $t=1$.
We use the same combinations of meshes and time steps as in Section~\ref{sec:tg}.
The results of this grid convergence study are presented in Tables~\ref{tab:gresho1}--\ref{tab:gresho3}.

\begin{table}[ht!]
\centering
\begin{tabular}{c||cc|cc||cc|cc}
$\sqrt2/h$ & $e_\mathbf{u}^C$ & EOC & $e_p^C$ & EOC & $e_\mathbf{u}^L$ & EOC & $e_p^L$ & EOC \\
\hline
16  & 5.92E-02 &      & 2.23E-02 &      & 5.01E-02 &      & 2.15E-02 &      \\
32  & 1.95E-02 & 1.60 & 6.40E-03 & 1.80 & 1.72E-02 & 1.54 & 6.09E-03 & 1.82 \\
64  & 7.02E-03 & 1.48 & 1.58E-03 & 2.02 & 5.55E-03 & 1.63 & 1.52E-03 & 2.00 \\
128 & 2.54E-03 & 1.47 & 3.82E-04 & 2.05 & 1.84E-03 & 1.59 & 3.74E-04 & 2.02 \\
256 & 9.67E-04 & 1.39 & 9.37E-05 & 2.03 & 6.56E-04 & 1.48 & 9.21E-05 & 2.02 \\
\hline
average && 1.49 && 1.98 && 1.56 && 1.97
\end{tabular}
\caption{L$^2(\Omega)$ convergence for the Gresho vortex on Friedrichs--Keller triangulations.}\label{tab:gresho1}
\end{table}

\begin{table}[ht!]
\centering
\begin{tabular}{c||cc|cc||cc|cc}
$1/h$ & $e_\mathbf{u}^C$ & EOC & $e_p^C$ & EOC & $e_\mathbf{u}^L$ & EOC & $e_p^L$ & EOC \\
\hline
10  & 1.13E-01 &      & 4.23E-02 &      & 5.67E-02 &      & 2.83E-02 &      \\
20  & 3.72E-02 & 1.60 & 1.02E-02 & 2.05 & 2.35E-02 & 1.27 & 8.77E-03 & 1.69 \\
40  & 1.37E-02 & 1.44 & 2.51E-03 & 2.02 & 8.57E-03 & 1.45 & 2.30E-03 & 1.93 \\
80  & 5.18E-03 & 1.40 & 7.18E-04 & 1.81 & 3.02E-03 & 1.51 & 5.97E-04 & 1.95 \\
160 & 1.78E-03 & 1.54 & 1.71E-04 & 2.07 & 1.13E-03 & 1.42 & 1.59E-04 & 1.91 \\
\hline
average && 1.50 && 1.99 && 1.41 && 1.87
\end{tabular}
\caption{L$^2(\Omega)$ convergence for the Gresho vortex on unstructured Delaunay meshes.}
\end{table}

\begin{table}[ht!]
\centering
\begin{tabular}{c||cc|cc||cc|cc}
$\sqrt2/h$ & $e_\mathbf{u}^C$ & EOC & $e_p^C$ & EOC & $e_\mathbf{u}^L$ & EOC & $e_p^L$ & EOC \\
\hline
16  & 4.34E-02 &      & 2.02E-02 &      & 3.99E-02 &      & 1.95E-02 &      \\
32  & 1.44E-02 & 1.59 & 5.55E-03 & 1.86 & 1.31E-02 & 1.61 & 5.41E-03 & 1.85 \\
64  & 5.32E-03 & 1.44 & 1.46E-03 & 1.93 & 4.76E-03 & 1.46 & 1.40E-03 & 1.94 \\
128 & 1.78E-03 & 1.58 & 3.69E-04 & 1.98 & 1.60E-03 & 1.58 & 3.61E-04 & 1.96 \\
256 & 6.81E-04 & 1.39 & 9.00E-05 & 2.03 & 5.91E-04 & 1.43 & 8.89E-05 & 2.02 \\
\hline
average && 1.50 && 1.95 && 1.52 && 1.94
\end{tabular}
\caption{L$^2(\Omega)$ convergence for the Gresho vortex on uniform quadrilateral grids.}\label{tab:gresho3}
\end{table}

Additionally, we track the kinetic energy evolution in pseudo-time for all runs.
The results shown in Fig.~\ref{fig:gresho-ene} are consistent with the findings from Section~\ref{sec:tg}.
While the exact energy is constant at steady state, all numerical approximations are dissipating energy.
This issue cannot be avoided because for nonconstant velocities (as in this test), some of the symmetric terms $G_{ij}$ that appear in \eqref{etadot2} are nonzero, and energy is therefore lost over time.
Increased resolution reduces these errors, and optimal convergence rates are attained for the velocity and pressure approximations.
Snapshots of the velocity magnitude and pressure at $t=1$ obtained on the finest Delaunay mesh are displayed in Fig.~\ref{fig:gresho}.

\begin{figure}[ht!]
\centering
\begin{subfigure}[b]{0.32\textwidth}
\caption{Friedrichs--Keller}
\includegraphics[width=\textwidth,trim=75 0 50 0,clip]{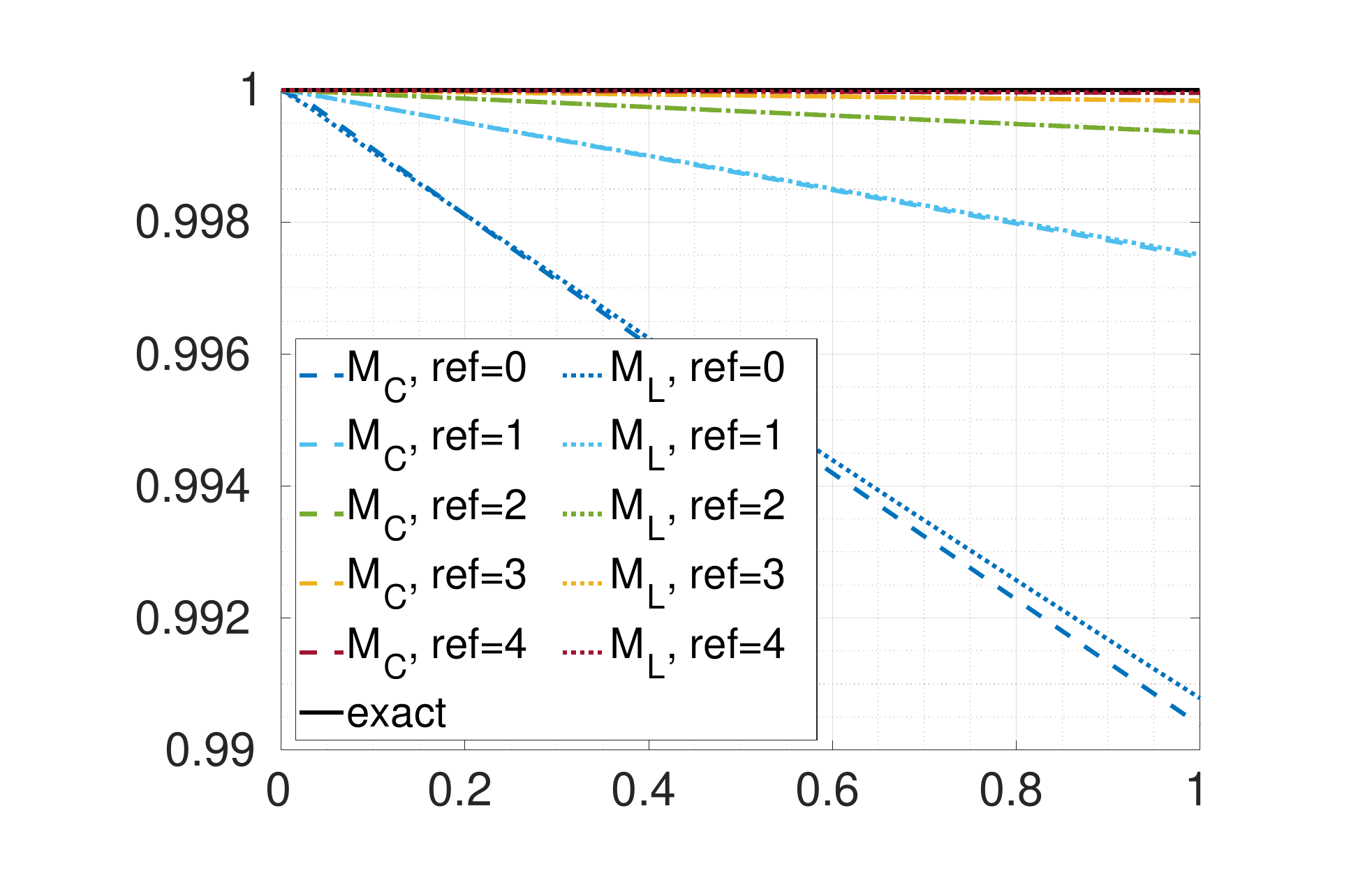}
\end{subfigure}
\begin{subfigure}[b]{0.32\textwidth}
\caption{Unstructured Delaunay}
\includegraphics[width=\textwidth,trim=75 0 50 0,clip]{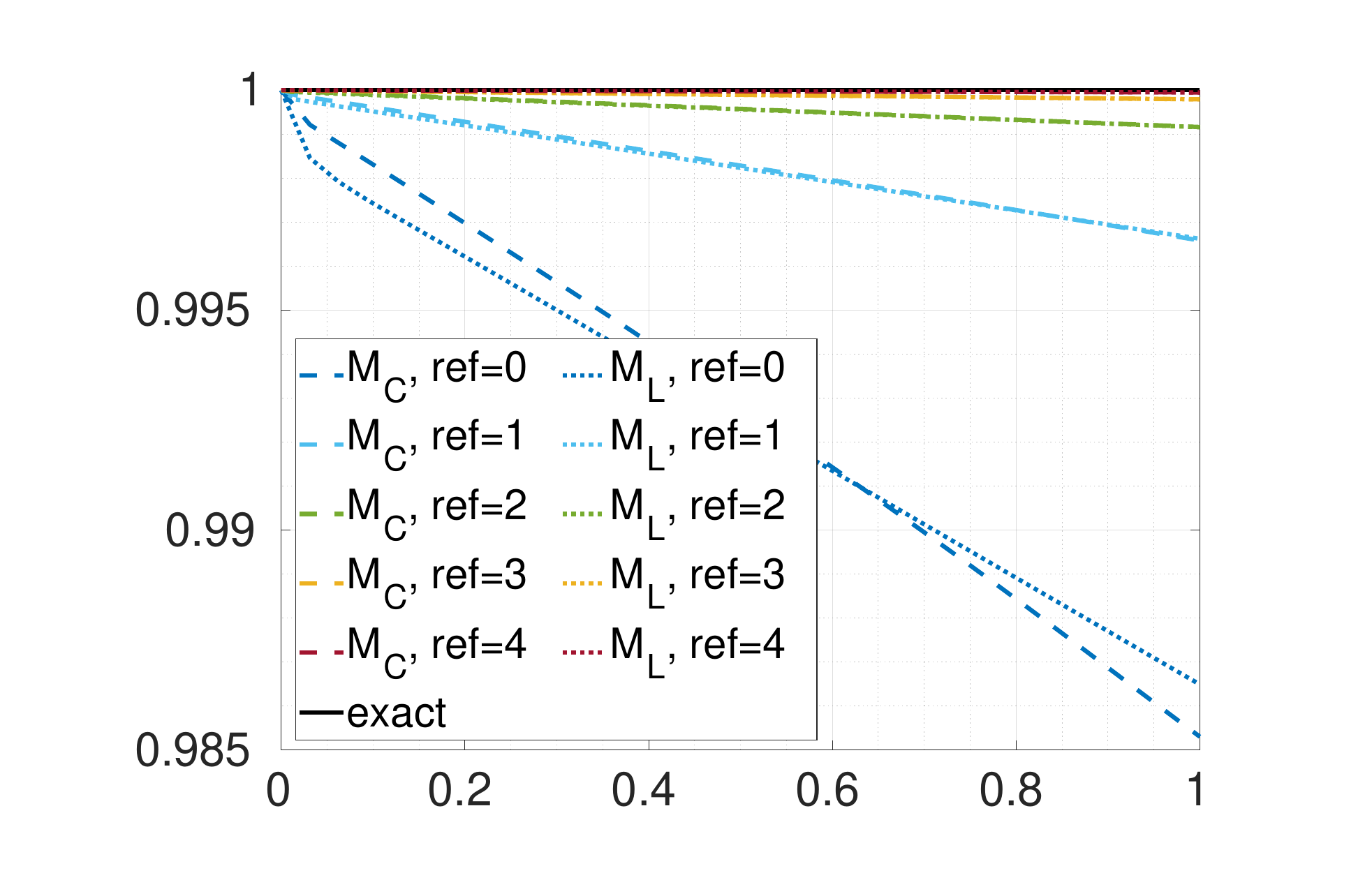}
\end{subfigure}
\begin{subfigure}[b]{0.32\textwidth}
\caption{Uniform quadrilateral}
\includegraphics[width=\textwidth,trim=75 0 50 0,clip]{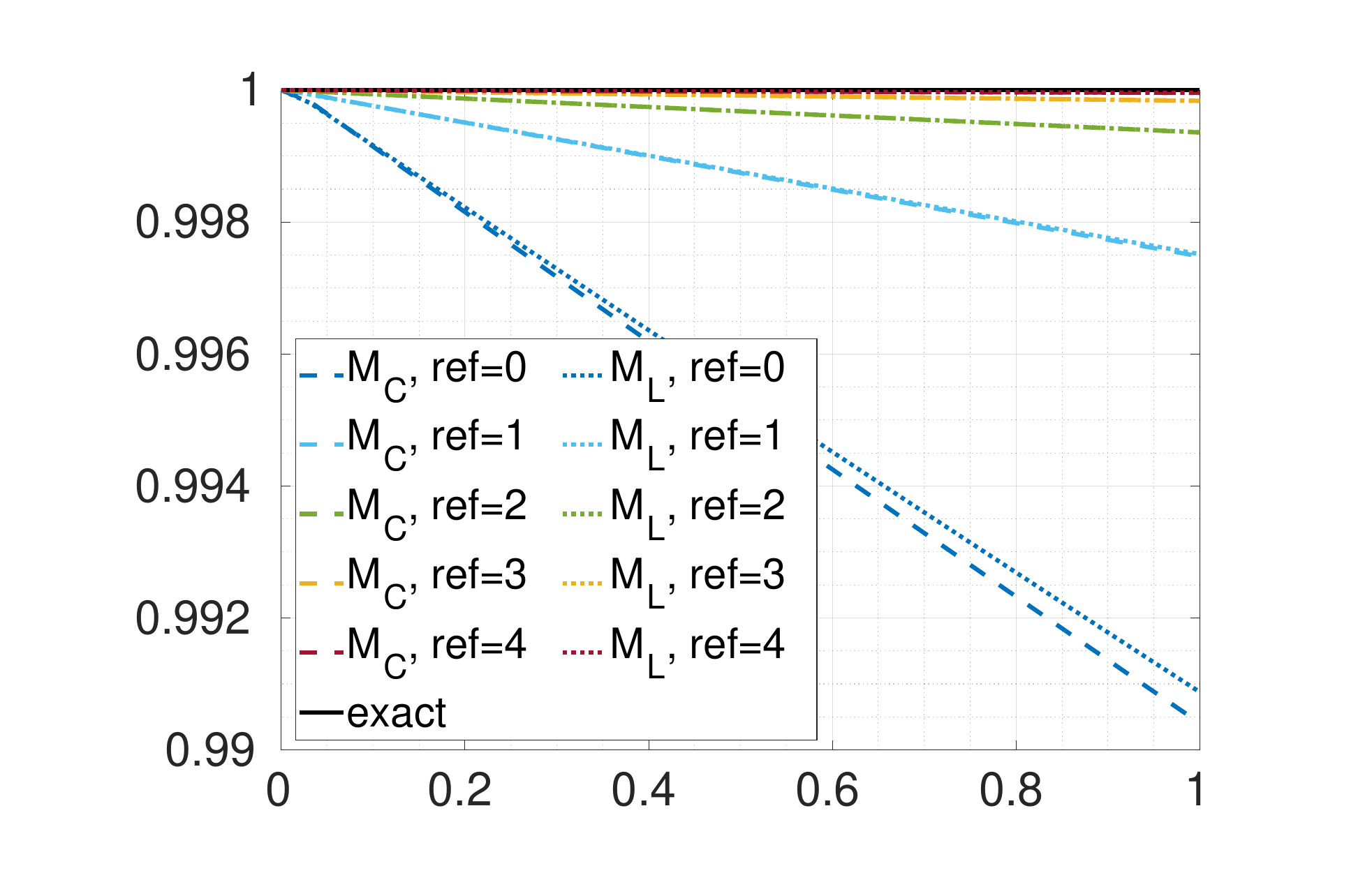}
\end{subfigure}
\caption{Energy evolutions for the Gresho vortex, exact profile and numerical results obtained on refinement levels \texttt{ref} for three types of meshes.}\label{fig:gresho-ene}
\end{figure}

\begin{figure}[ht!]
\centering
\begin{subfigure}[b]{0.24\textwidth}
\caption{$|\mathbf u|\in{}$[3.0E-6,0.99]}
\includegraphics[width=\textwidth]{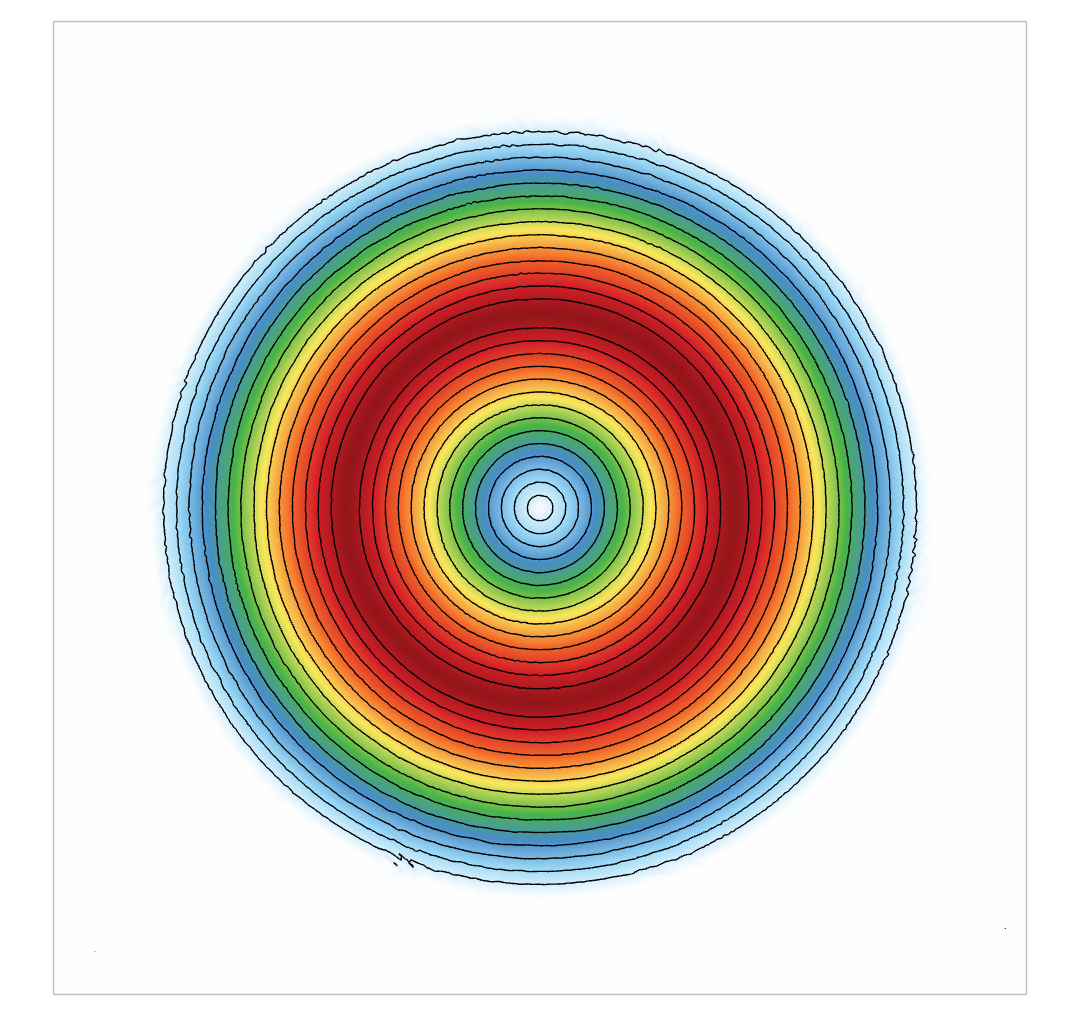}
\end{subfigure}
\begin{subfigure}[b]{0.24\textwidth}
\caption{$|\mathbf u|\in{}$[1.0E-6,0.99]}
\includegraphics[width=\textwidth]{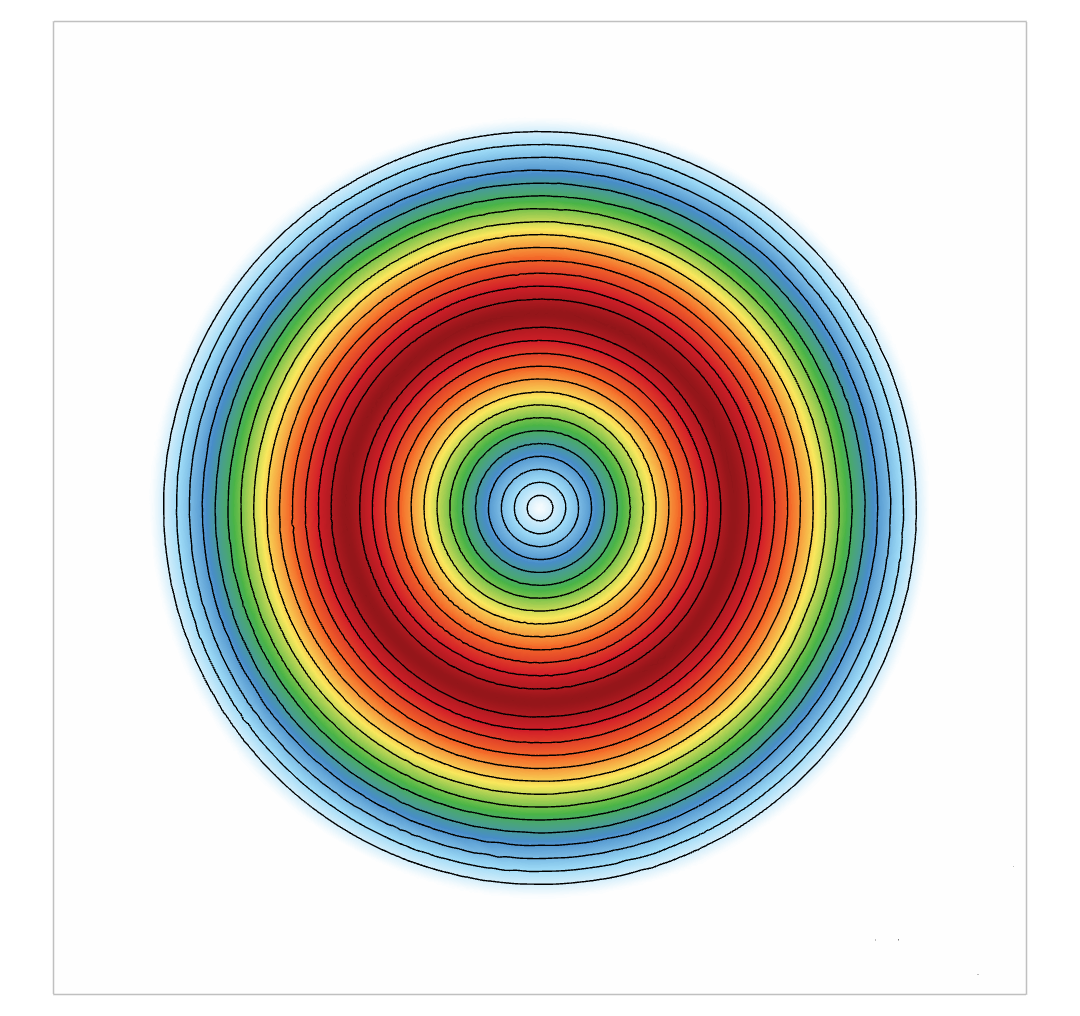}
\end{subfigure}
\begin{subfigure}[b]{0.24\textwidth}
\caption{$p\in{}$[-0.69,0.084]}
\includegraphics[width=\textwidth]{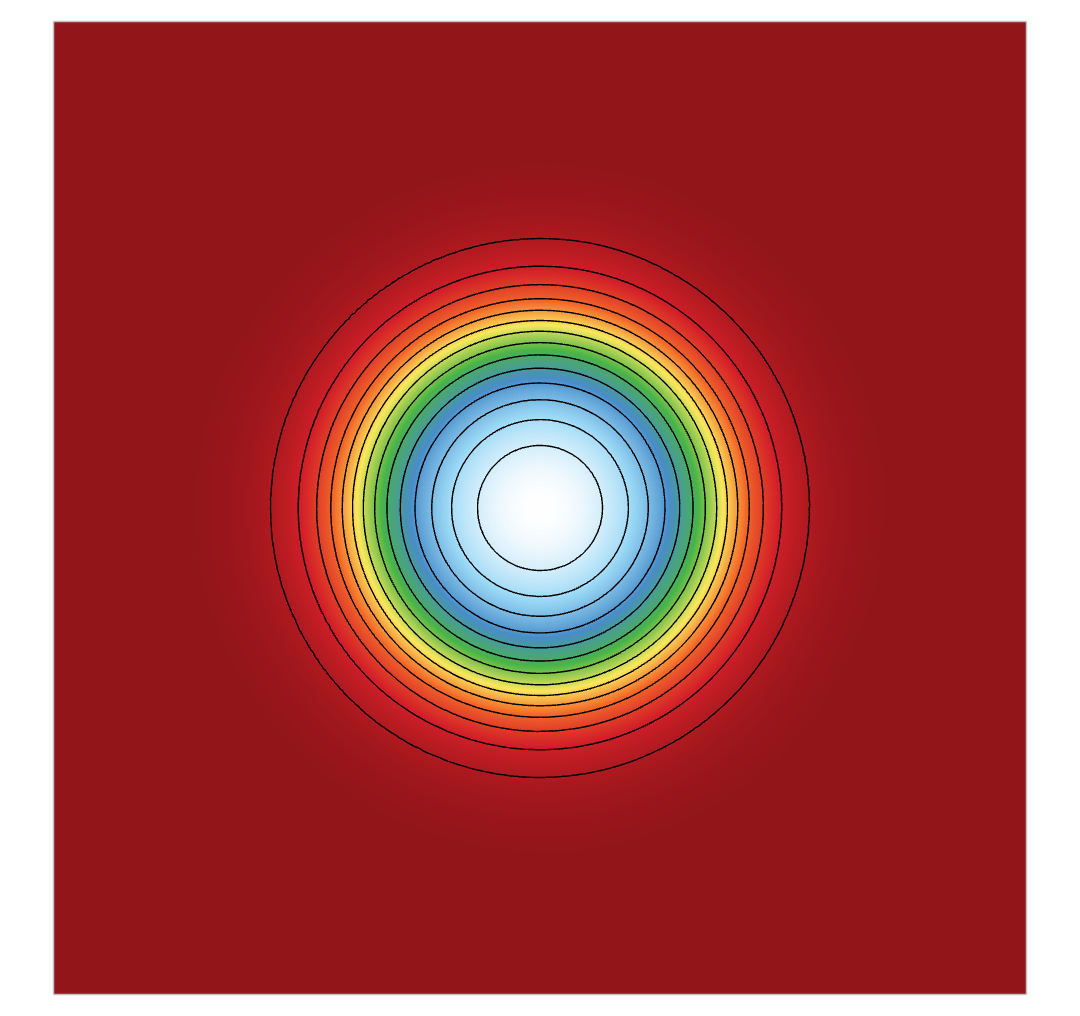}
\end{subfigure}
\begin{subfigure}[b]{0.24\textwidth}
\caption{$p\in{}$[-0.69,0.084]}
\includegraphics[width=\textwidth]{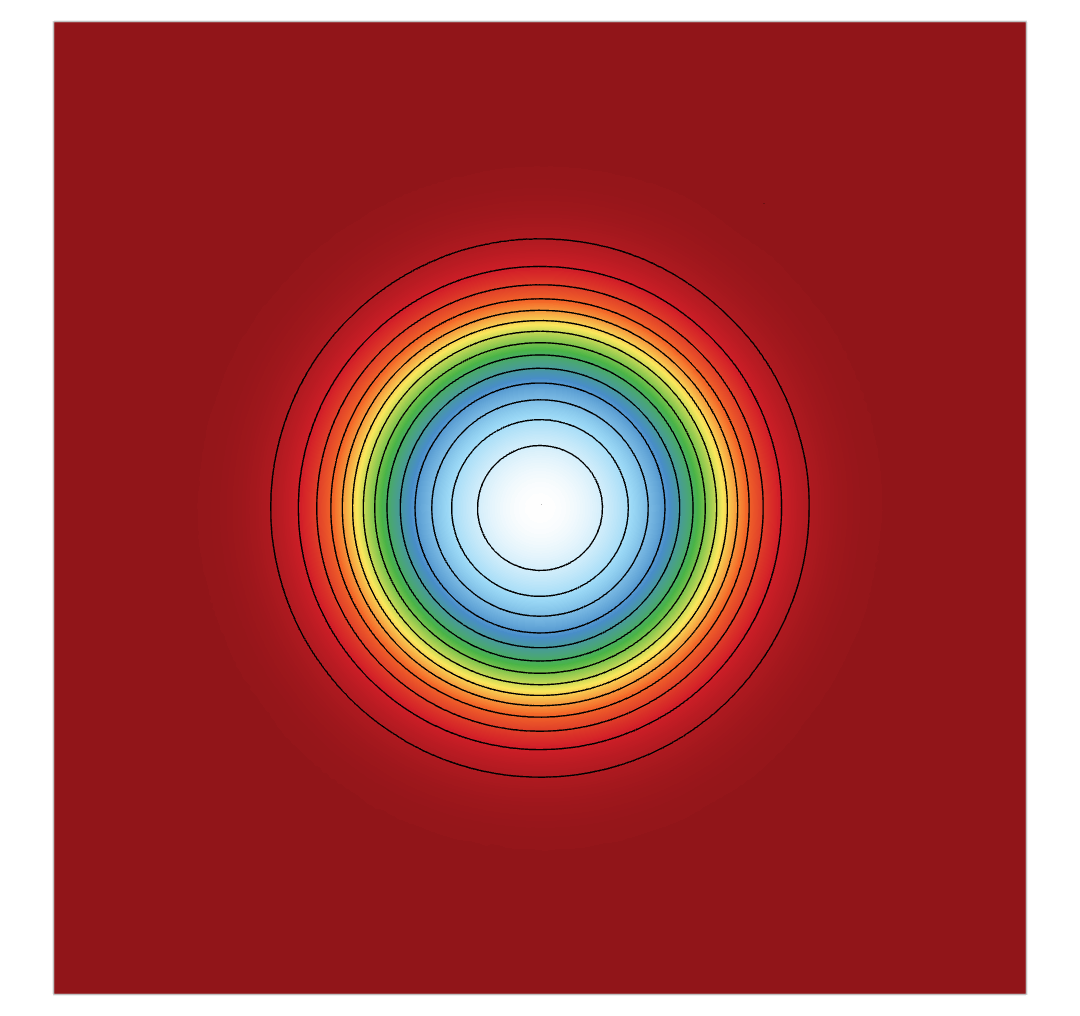}
\end{subfigure}
\caption{Finite element approximations to $|\mathbf u|$ and $p$ for the Gresho vortex at $t=1$ obtained on the finest level of the unstructured Delaunay~mesh using the consistent (panels a,c) and lumped (panels b,d) mass matrices.}\label{fig:gresho}
\end{figure}

The Gresho vortex is also a perfect example for illustrating the necessity of using stabilization techniques.
To show the superiority of \eqref{galerkin-stab} over \eqref{galerkin}, we ran additional tests for a combination of the Galerkin space discretization with implicit-explicit (IMEX) time stepping of Euler type.
In the so-defined fully discrete scheme, the nonlinear convective term
$
(\varphi_i,\mathbf u_h \cdot \nabla\mathbf u_h)_{L^2(\Omega)}
$
is treated explicitly, as in the forward Euler version, while all other terms are treated implicitly, as in the backward Euler version.
In our numerical studies for the Galerkin IMEX-Euler method with consistent and lumped mass matrices, we use a coarse Friedrichs--Keller triangulation with $h=\sqrt 2 /32$ in combination with time steps $\Delta t = 0.5^k$, where $k\in \{6,8,10,12\}$.
Note that for the finest time step, the ratio $\Delta t/h\approx 0.0055$ is already extremely small.
The energy curves for this experiment are displayed in Fig.~\ref{fig:gresho-coarse} along with the discrete initial distribution of the velocity magnitude.
For each of the eight runs, there exists a time at which kinetic energy starts growing and, as a consequence, numerical instabilities arise.
We observe severe blowups, as the energy quickly tends to infinity, thus rendering the corresponding approximations useless.
As before, lumping of mass matrices is beneficial for the unstable $\mathbb P_1\mathbb P_1$ discretization, which manifests itself in the fact that the instabilities occur later than in the corresponding consistent mass matrix runs.
In our experience, breakdowns and blowups due to spurious energy production can only be delayed to some extent but not avoided altogether without energy stabilization.

\begin{figure}[ht!]
\centering
\begin{subfigure}[t]{0.24\textwidth}
\caption{$|\mathbf u_0|\le 0.9598$}
\includegraphics[width=\textwidth]{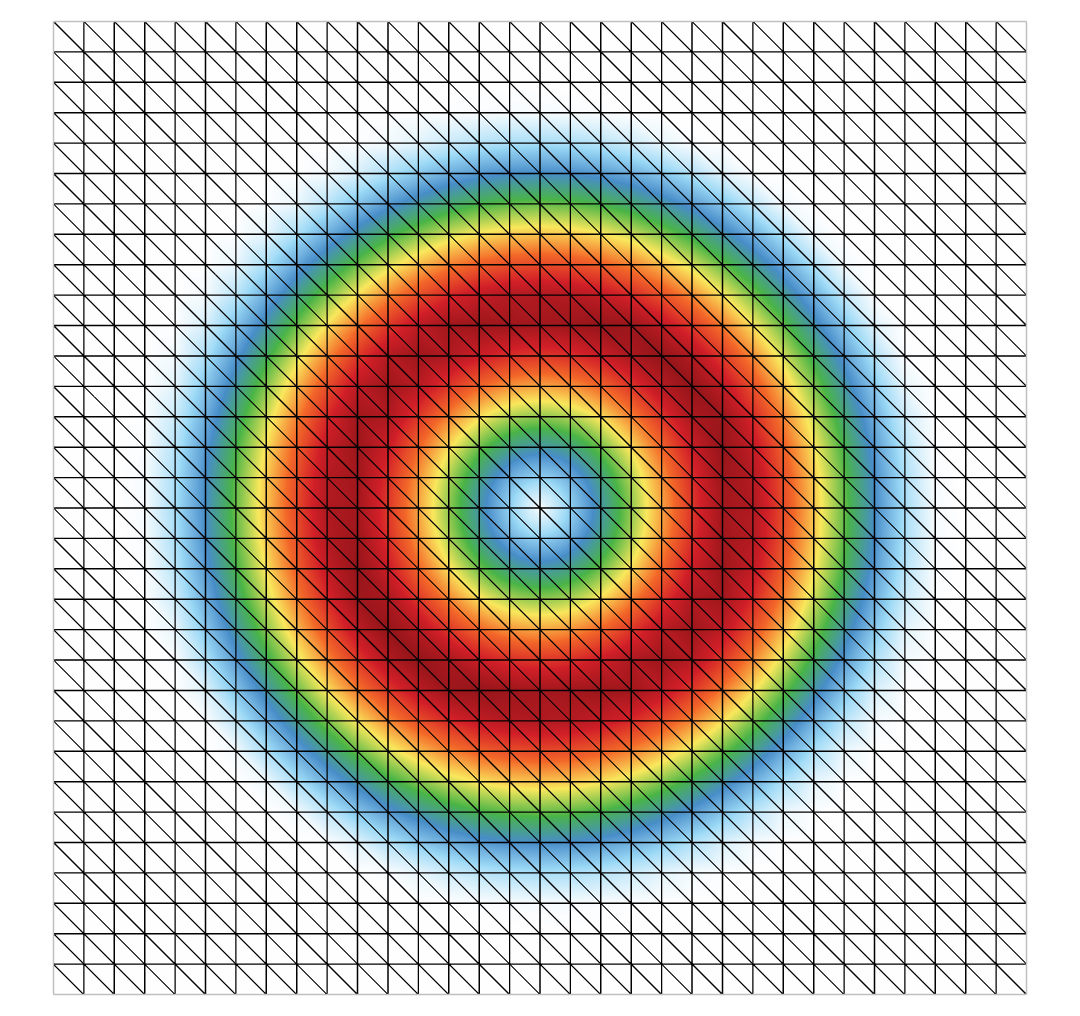}
\end{subfigure}
\begin{subfigure}[t]{0.37\textwidth}
\caption{Consistent mass matrix}
\includegraphics[width=\textwidth]{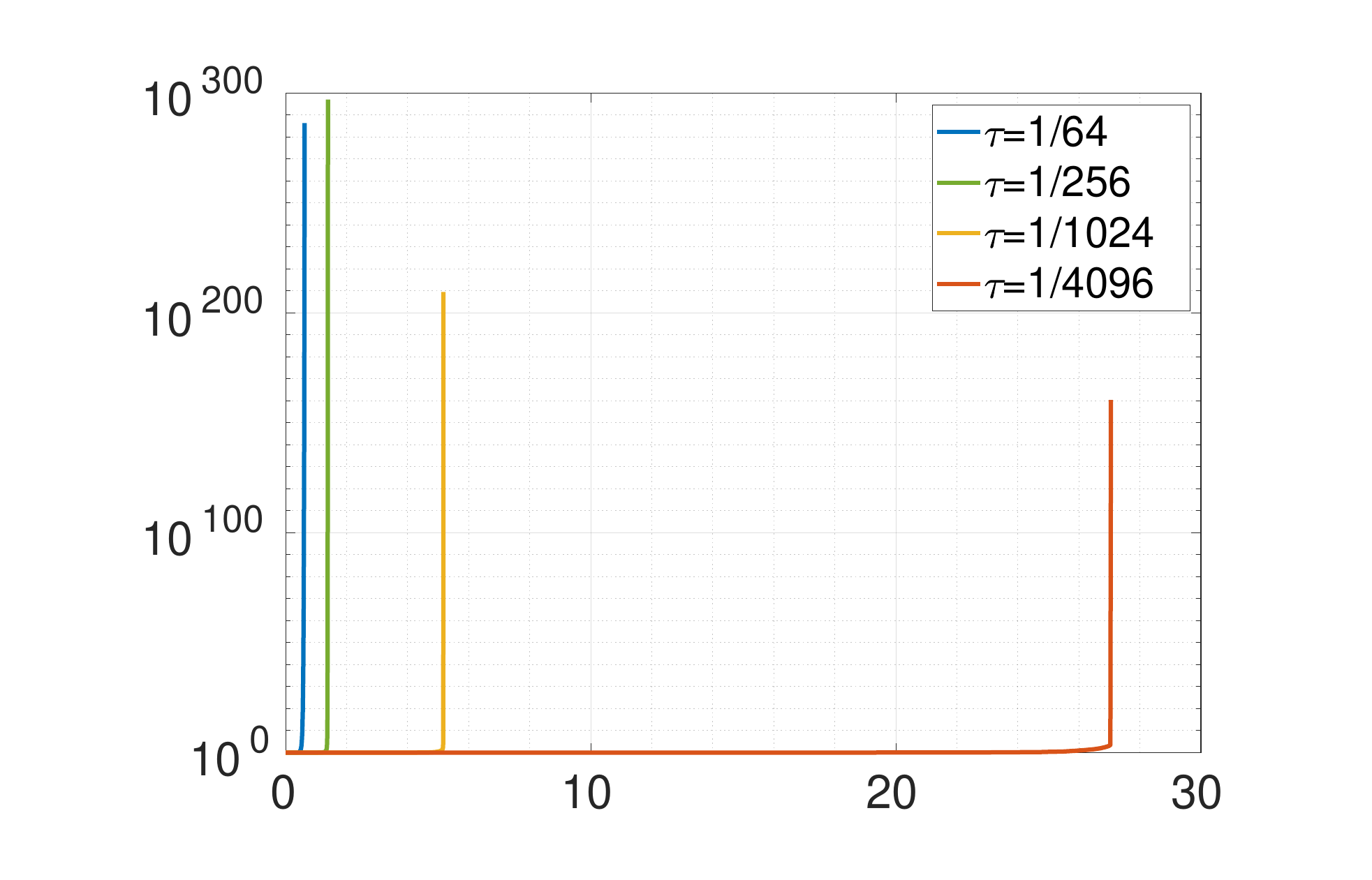}
\end{subfigure}
\begin{subfigure}[t]{0.37\textwidth}
\caption{Lumped mass matrix}
\includegraphics[width=\textwidth]{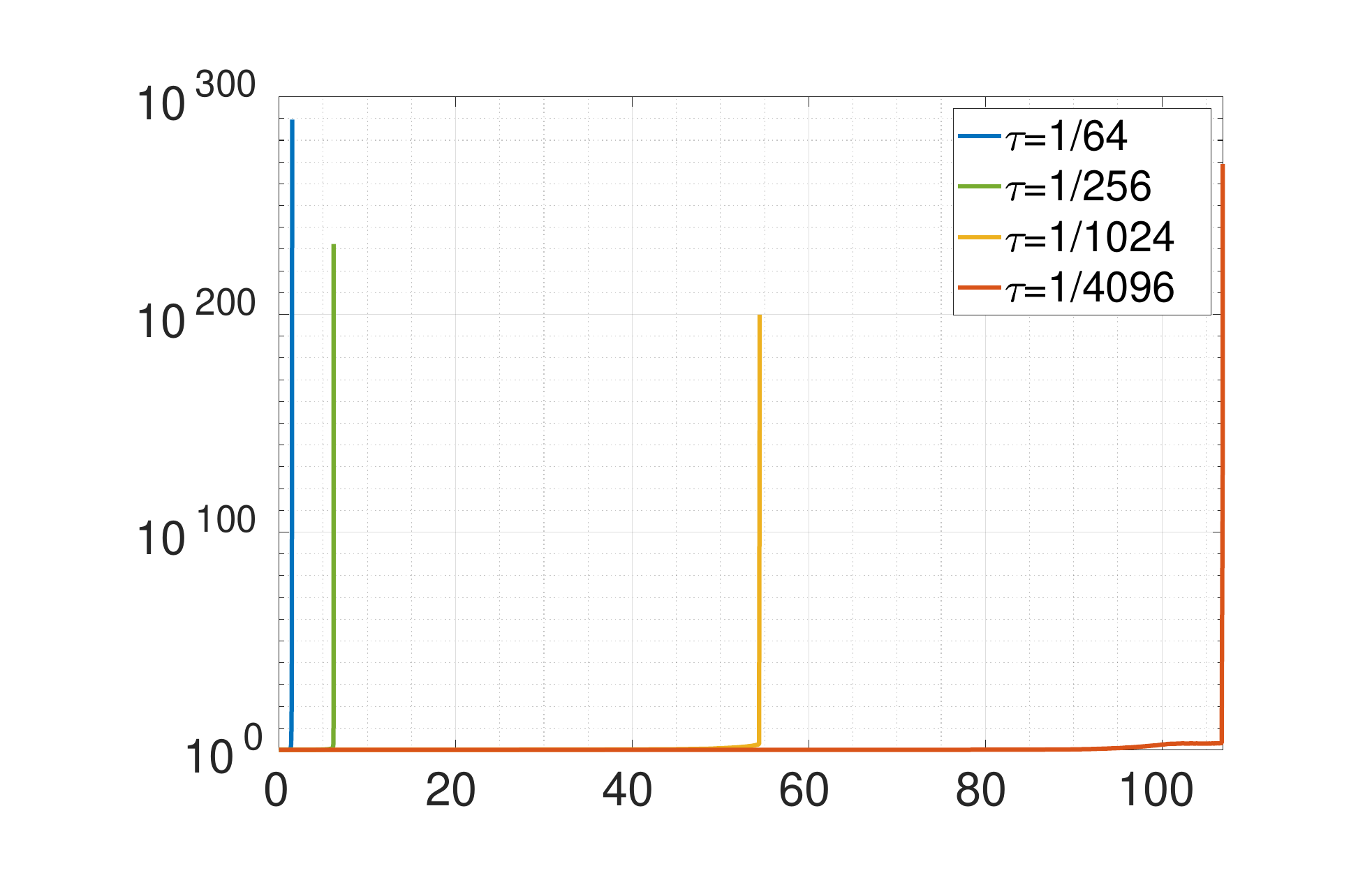}
\end{subfigure}
\caption{Gresho vortex results for the non-stabilized Galerkin IMEX-Euler discretization on a Friedrichs--Keller triangulation with $h=\sqrt 2/32$.
The diagrams show a mesh plot of the discrete initial data (panel a) and semi-logarithmic plots of energy evolutions for simulations using the consistent (panel b) and lumped (panel c) mass matrices.}\label{fig:gresho-coarse}
\end{figure}

In contrast, our new energy-stable schemes can be run with arbitrarily large time steps and still produce stable results.
The use of large pseudo-time steps $\Delta t$ yields harder-to-solve linear systems because small time steps lead to dominant mass matrix contributions.
However, fewer updates produce less diffusive approximations to steady-state solutions.
To quantify these effects, we ran the Gresho test again up to end time $t=1000$ with all three types of meshes using $h=1/128$ for the structured grids and $h=1/80$ in the case of the unstructured Delaunay triangulation.
In Tables~\ref{tab:gresho-long1}--\ref{tab:gresho-long1000}, we display the total time required for computations using $\Delta t \in \{1,1000\}$ and both kinds of mass matrices.
To give an indication of how diffusive each approach is, we also report the maximum velocity magnitude and relative energy loss.
Note that $\Delta t=1$ is already larger by a factor of 256 than the corresponding time step employed in the earlier convergence studies for the same grids.
For accuracy reasons, the use of time steps as large (compared to the mesh size) as in the present experiments is naturally discouraged for transient problems.
Our only reason for including the elapsed times here is to facilitate a comparison of the linear solver's performance for $\Delta t=1$ vs.\ $\Delta t=1000$.
Thus, we refrain from repeating various runs with the same numerical parameters or testing different hardware.

Despite the fact that the time step $\Delta t=1000$ is extremely large, the corresponding approximations on structured grids agree qualitatively well with the ones shown in Fig.~\ref{fig:gresho}.
On the unstructured mesh, the pressure solution is in reasonable agreement with the exact solution, while the velocity field contains nonphysical artifacts.
All results obtained with $\Delta t=1$ are useless because temporal discretization errors of order $\mathcal O(\Delta t^3) = \mathcal O(1)$ arise in each time step and accumulate over time.
The resulting velocity and pressure fields are oscillatory and bear no resemblance to the exact solution.
However, the proposed scheme remains stable with large percentages of energy being dissipated.

The large time step experiments for the Gresho vortex are meant to demonstrate the stability of our new scheme and the performance of the linear solvers in the context of steady-state computations.
The occurrence of bounded spurious oscillations does not contradict the theory developed in Sections~\ref{sec:strong} and \ref{sec:weak} for the energy-stable space discretization \eqref{galerkin-stable}.
The use of very large time steps in fully discrete schemes may result in numerical artifacts due to unacceptable rates of local energy production.
However, no increase in the total energy was observed in any of our numerical experiments.

Slight differences in the computational time and diffusivity for discretizations listed in Tables~\ref{tab:gresho-long1} and~\ref{tab:gresho-long1000} are due to varying spatial resolution.
It is interesting to note that in the experiment using $\Delta t=1000$, the lumped-mass version is more efficient for triangular grids than for quadrilaterals.
Moreover, Friedrichs--Keller triangulations seem to outperform quadrilateral grids despite employing the same number of unknowns and yielding results of comparable accuracy.
These cost factors can be attributed to the linear solver (MATLAB's backslash operator), which we found to be the bottleneck that dominates the computational cost of our current implementation.
The computational time for a simulation run using 1000 time steps by far exceeds the cost of a computation using a single time step.
Clearly, this is due to the need for solving many linear systems instead of just one.
Normalized by the factor 1000, however, the computational cost per time step is clearly much less for $\Delta t=1$.
This confirms that the use of inordinately large time steps may significantly degrade the conditioning of system matrices and increase the overall computational effort for solving linear systems.
Moreover, we ran into memory issues not encountered for $\Delta t=1$ when we tried to use $\Delta t=1000$ on even finer~grids.

\begin{table}[ht!]
\centering
\begin{tabular}{l|rrr}
Discretization & time [s] & $\max\limits_{\mathbf x\in\bar\Omega}|\mathbf u_h(\mathbf x,t=1000)|$ & energy loss \% \\
\hline
Friedrichs--Keller, consistent    & 655 & 0.947 & 80.9 \\
Friedrichs--Keller, lumped        & 655 & 0.690 & 72.7 \\
unstructured Delaunay, consistent & 576 &  1.22 & 77.2 \\
unstructured Delaunay, lumped     & 577 & 0.864 & 70.5 \\
uniform quadrilateral, consistent & 710 &  1.48 & 72.6 \\
uniform quadrilateral, lumped     & 722 & 0.794 & 64.0
\end{tabular}
\caption{Gresho vortex problem solved up to $t=1000$ using $\Delta t=1$.}\label{tab:gresho-long1}
\end{table}

\begin{table}[ht!]
\centering
\begin{tabular}{l|rrr}
Discretization & time [s] & $\max\limits_{\mathbf x\in\bar\Omega}|\mathbf u_h(\mathbf x,t=1000)|$ & energy loss \% \\
\hline
Friedrichs--Keller, consistent    & 41.1 & 0.930 & 7.66 \\
Friedrichs--Keller, lumped        & 21.3 & 0.929 & 7.65 \\
unstructured Delaunay, consistent & 21.0 &  3.82 & 9.28 \\
unstructured Delaunay, lumped     & 8.72 &  2.92 & 9.06 \\
uniform quadrilateral, consistent & 62.9 & 0.925 & 7.66 \\
uniform quadrilateral, lumped     & 64.0 & 0.926 & 7.65
\end{tabular}
\caption{Gresho vortex problem solved up to $t=1000$ using $\Delta t=1000$.}\label{tab:gresho-long1000}
\end{table}

\section{Conclusions}
\label{sec:conclusions}

The main focus of this work was on the theoretical and practical
implications of local energy inequalities in the context of (continuous
$\mathbb{P}_1\mathbb{P}_1$) finite element
methods for incompressible flow problems. The proposed methodology has a
lot in common with algebraic flux correction schemes for nonlinear hyperbolic
systems, which enables us to exploit the existing similarities. The results
of our theoretical studies and numerical experiments illustrate the importance of
local energy stability, which we use as a tool for proving consistency
and convergence towards dissipative weak solutions. It is hoped that our
findings provide strong motivation for further research efforts in
the field of structure-preserving finite element schemes for the 
incompressible Navier--Stokes equations. We finally remark that the design
of locally energy stable high-order extensions is likely to require subcell
flux correction (as in \cite{hajduk2021,kuzmin2020a}) and more
sophisticated analysis or the use of alternative energy stabilization
procedures.
 
\section*{Acknowledgments}

The work of D.K. and P.\"O.
was supported by the German Research Foundation (DFG) within the framework of the priority research program SPP 2410 under grants KU 1530/30-1 and OE 661/5-1, respectively.
P.\"O. was also supported by the DFG under the personal grant 520756621 (OE 661/4-1).

H.H. acknowledges support from The Rough Ocean Project, funded by the Research Council of Norway under the Klimaforsk-programme, Project 302743.
H.H. also thanks Dr. Christoph Lohmann (TU Dortmund University) for a helpful suggestion on how to efficiently assemble certain operators.

\bibliographystyle{abbrv}
\bibliography{literature}

\end{document}